\documentclass[11pt,leqno]{article}
\usepackage{graphicx, amsfonts, amsthm, amsxtra, amssymb, verbatim,latexsym}
\usepackage{hyperref}
\usepackage[mathscr]{euscript}
\usepackage[all,cmtip]{xy}

\begin{document}
%\hypersetup{
%    colorlinks=true, %set true if you want colored links
%    linktoc=all,     %set to all if you want both sections and subsections linked
%    linkcolor=blue,  %choose some color if you want links to stand out
%}
%\input{notations.tex}
%%%%%%%%%%%%%%%%%%%%%%%%%%%%%%%%%%%%%%%%%%%%%%%%%%%%%%5
%\include{notations}

%----------------General Math Notations-------------------------------------
\def\Z{\mathbb{Z}}                   %Integer  numbers
\def\Q{\mathbb{Q}}                   %Rational  numbers
\def\C{\mathbb{C}}                   %Complex numbers
\def\N{\mathbb{N}}                   %natural numbers
\def\uhp{{\mathbb H}}                %upper half plane
\def\A{\mathbb{A}}                   %affine space C^n
\def\dR{{\rm dR}}                    %The subindex dR standing for de Rham cohomology.
\def\F{{\cal F}}                     %A foliation
\def\Sp{{\rm Sp}}                    %Symplectic group
\def\Gm{\mathbb{G}_m}                 %The multiplicative group
\def\Ga{\mathbb{G}_a}                 %The additive  group
\def\Tr{{\rm Tr}}                      %Trace map in Algebraic de Rham cohomology
\def\tr{{{\mathsf t}{\mathsf r}}}                 %Transposition of matrices
\def\spec{{\rm Spec}}            %The spectrume
\def\ker{{\rm ker}}              %kernel
\def\GL{{\rm GL}}                %The liner group
%----------------Gauss-Manin Connection in disguise---------------------
\def\k{{\sf k}}                     %Arbitrary field
\def\ring{{ R}}                   %A ring
\def\X{{\sf X}}                      %families of varieties
\def\T{{\sf T}}                      %Moduli of enhanced  varieties
\def\Ts{{\sf S}}
\def\cmv{{\sf M}}                    %Classical moduli of varieties
\def\BG{{\sf G}}                       %Borel Algebraic Group
\def\podu{{\sf pd}}                   %Poincare dual
\def\ped{{\sf U}}                    %Period Domain
\def\per{{\sf  P}}                   %period matrix or fundamental system of GMC
\def\gm{{\sf  A}}                    %Gauss-Manin connection
\def\gma{{\sf  B}}                   %Gauss-Manin connection
\def\ben{{\sf b}}                    %Betti number

\def\Rav{{\mathfrak M }}                     % Space of modular vector fields
\def\Ram{{\mathfrak C}}                     % Space of constant vector fields
\def\Rap{{\mathfrak G}}                     % Space of vector fields arising from group  action.

\def\nov{{\sf  n}}                    %Dimension of the variety
\def\mov{{\sf  m}}                    %Fixed dimension of the cohomology
\def\Yuk{{\sf Y}}                     %Yukawa and generalizations
\def\Ra{{\sf R}}                      %Ramanujan type vector field
\def\hn{{ h}}                      %Hodge numbers
\def\cpe{{\sf C}}                     %Constant periods
\def\g{{\sf g}}                       %An element of the Borel group
\def\t{{\sf t}}                       %An element \T
\def\pedo{{\sf  \Pi}}                  %Period domain before discrete group action

\def\Der{{\rm Der}}                   %Derivations
\def\MMF{{\sf MF}}                    %Moduli of modular foliations
\def\codim{{\rm codim}}                %codimension
\def\dim{{\rm    dim}}                %dimension
\def\Lie{{\rm Lie}}                   %Lie algebra of a group.
\def\gg{{\mathfrak g}}                %elements of a Lie group

\def\u{{\sf u}}                       %An element of the period domain

\def\imh{{  \Psi}}                 %intersection matrix in homology
\def\imc{{  \Phi }}                  %intersection matrix in cohomology
\def\stab{{\rm Stab }}               %stablizer
\def\Vec{{\rm Vec}}                 %space of vector fields
\def\prim{{\rm prim}}                  %primitive cohomology

\def\Fg{{\sf F}}     %Genus g topological partition function
\def\hol{{\rm hol}}  %holomorphic
\def\non{{\rm non}}  %non-holomorphic
\def\alg{{\rm alg}}  %algebraic

\def\bcov{{\rm \O_\T}}       %The ring of modular-type functions

\def\leaves{{\cal L}}        %space of leaves

%---------------OLD Notation-------------------
\def\GM{{\rm GM}}

\def\perr{{\sf q}}        %period matrix.....
\def\perdo{{\cal K}}   %period domain
\def\sfl{{\mathrm F}} %Space of filtrations
\def\sp{{\mathbb S}}  %Sphere

\newcommand\diff[1]{\frac{d #1}{dz}} %Differential operator
\def\End{{\rm End}}              %Endomorphism group

\def\sing{{\rm Sing}}            %The set of singularities
\def\cha{{\rm char}}             %Charracteristic
\def\Gal{{\rm Gal}}              %The Galois group
\def\jacob{{\rm jacob}}          %the Jacobian ideal
\def\tjurina{{\rm tjurina}}      %the tjurina ideal
\newcommand\Pn[1]{\mathbb{P}^{#1}}   %Projective space of dimension #1
\def\Ff{\mathbb{F}}                  %Finite field

\def\O{{\cal O}}                     %ring of integers of a number field
\def\as{\mathbb{U}}                  %Some affine space
\def\ring{{ R}}                         %A ring
\def\R{\mathbb{R}}                   %real numbers

\newcommand\ep[1]{e^{\frac{2\pi i}{#1}}}% unipotent numbers
\newcommand\HH[2]{H^{#2}(#1)}        %Hodge structures
\def\Mat{{\rm Mat}}              %Matrices
\newcommand{\mat}[4]{
     \begin{pmatrix}
            #1 & #2 \\
            #3 & #4
       \end{pmatrix}
    }                                %two by two matrices
\newcommand{\matt}[2]{
     \begin{pmatrix}                 % one by two matrix
            #1   \\
            #2
       \end{pmatrix}
    }
\def\cl{{\rm cl}}                %Chern class

\def\hc{{\mathsf H}}                 %The set of Hodge cycles.
\def\Hb{{\cal H}}                    %Hodge bundle
\def\pese{{\sf P}}                  %Period set

\def\PP{\tilde{\cal P}}              %the period domain/ discrete group
\def\K{{\mathbb K}}                  %Field representing R or C

\def\M{{\cal M}}
\def\RR{{\cal R}}
\newcommand\Hi[1]{\mathbb{P}^{#1}_\infty}%the hyperplane at infinity
\def\pt{\mathbb{C}[t]}               %Polynomials in t
\def\W{{\cal W}}                     %weight filtration
\def\gr{{\rm Gr}}                %graded pieces
\def\Im{{\rm Im}}                %imaginary
\def\Re{{\rm Re}}                %Real
\def\depth{{\rm depth}}
\newcommand\SL[2]{{\rm SL}(#1, #2)}    %SL(2,Z)
\newcommand\PSL[2]{{\rm PSL}(#1, #2)}  %PSL(2,Z)
\def\Resi{{\rm Resi}}              %Residue

\def\L{{\cal L}}                     %The moduli of polarized lattices in a
                                     %fixed vector spaces.
\def\Aut{{\rm Aut}}              %Automorphism group of a vectorspace
\def\any{R}                          %Any subring of the field of complex
                                     %numbers.
\newcommand\ovl[1]{\overline{#1}}    %Conjugation of #1.

\newcommand\mf[2]{{M}^{#1}_{#2}}     %New modular functions
\newcommand\mfn[2]{{\tilde M}^{#1}_{#2}}     %New modular functions

\newcommand\bn[2]{\binom{#1}{#2}}    %Binomial
\def\ja{{\rm j}}                 %j of a two by two matrix
\def\Sc{\mathsf{S}}                  %Simple cycles
\newcommand\es[1]{g_{#1}}            %Eisenstein series
\newcommand\V{{\mathsf V}}           %Milnor vector space
\newcommand\WW{{\mathsf W}}          %Similar to Milnor vector space
\newcommand\Ss{{\cal O}}             %Structural sheaf
\def\rank{{\rm rank}}                %rank of a module
\def\Dif{{\cal D}}                   %Differentials
\def\gcd{{\rm gcd}}                  %greatest common divisor
\def\zedi{{\rm ZD}}                  %zero divisors of a module
\def\BM{{\mathsf H}}                 %Brieskorn module
\def\plf{{\sf pl}}                             %Picard-Lefschetz formula
\def\sgn{{\rm sgn}}                      %sign
\def\diag{{\rm diag}}                   %diagonal matrix
\def\hodge{{\rm Hodge}}
\def\HF{{\sf F}}                                %The hodge filtration of the brieskon module
\def\WF{{\sf W}}                               %The weight filtration of the brieskon module
\def\HV{{\sf HV}}                                %humbert variety
\def\pol{{\rm pole}}                               %pole divisor
\def\bafi{{\sf r}}
\def\id{{\rm id}}                               %identity
\def\gms{{\sf M}}                           %Gauss-Manin system
\def\Iso{{\rm Iso}}                           %Gauss-Manin system

\def\hl{{\rm L}}    %holomorphic limit
\def\imF{{\rm F}}
\def\imG{{\rm G}}
%-----------------------Additional notations----------------------
\def\cy{{Calabi-Yau }}
\def\DHR{{\rm DHR }}            %DHR vector field
\def\H{{\sf H}}            %DHR vector field
\def\P {\mathbb{P}}                  %Projective Space
\def\E{{\cal E}}
\def\L{{ L}}             %PF equation
\def\CX{{\cal X}}
\def\dt{{\sf d}}             %dimension of enhanced moduli space T

\def\gmo{{\sf  B}}                    %Gauss-Manin connection

%%%%%%%%%%%%%%%%%%%%%%%%%%%%%%%%%%%%%%%%%%%%%%%%%%%%%%%%5
%---------------------------------------------
\newtheorem{theo}{Theorem}[section]
\newtheorem{exam}{Example}[section]
\newtheorem{coro}{Corollary}[section]
\newtheorem{defi}{Definition}[section]
\newtheorem{prob}{Problem}[section]
\newtheorem{lemm}{Lemma}[section]
\newtheorem{prop}{Proposition}[section]
\newtheorem{rem}{Remark}[section]
\newtheorem{obs}{Observation}[section]
\newtheorem{conj}{Conjecture}
\newtheorem{nota}{Notation}[section]
\newtheorem{ass}{Assumption}[section]
\newtheorem{calc}{}
\numberwithin{equation}{section}

\begin{center}
 {\LARGE\bf  Gauss-Manin connection in disguise: Dwork family}
\footnote{ MSC2010:
14J15,     %Moduli, classification: analytic theory; relations with modular forms
14J32,     %Calabi-Yau manifolds
%34M45,     %Differential equations on complex manifolds
%14N35,     %Gromov-Witten invariants, quantum cohomology,Gopakumar-Vafa invariants, Donaldson-Thomas invariants
11Y55.      %Calculation of integer sequences
\\
Keywords: Gauss-Manin connection, Dwork family, Picard-Fuchs
equation, Hodge filtration, quasi-modular form, q-expansion. }
%}
\\

\vspace{.25in} {\large {\sc H. Movasati \footnote{Instituto de
Matem\'atica Pura e Aplicada (IMPA), Rio de Janeiro, Brazil. email:
hossein@impa.br} and Y. Nikdelan \footnote{Universidade do Estado do
Rio de Janeiro (UERJ), Instituto de Matem\'atica
e Estat\'{i}stica (IME), Departamento de An\'alise Matem\'atica, Rio de Janeiro, Brazil. e-mail: younes.nikdelan@ime.uerj.br}}} \\
\end{center}
%%%%%%%%%%%%%%%%%%%%%%%%%%%%%%%%%%%%%%%%%%%%%%%%%%%%%%%%%%%%%%%%%%%%%%%%%%%
\vspace{.25in}
\begin{abstract}
We study the moduli space $\textsf{T}$ of the Calabi-Yau
$n$-folds  arising from the Dwork family and enhanced with bases of the $n$-th de Rham cohomology with constant cup product and compatible with Hodge filtration. We also describe a unique vector
field $\textsf{R}$ in $\textsf{T}$ which contracted with the Gauss-Manin connection gives an upper triangular matrix with some non-constant entries which are natural generalizations of Yukawa couplings.  For $n=1,2$ we
compute explicit expressions of $\textsf{R}$ and give a solution of
$\textsf{R}$ in terms of quasi-modular forms. The moduli space $\T$ is an affine variety and for $n=4$ we  give explicit coordinate system for $\T$ and compute the vector field $\textsf{R}$ and the $q$-expnasion of its solution.
\end{abstract}
%%%%%%%%%%%%%%%%%%%%%%%%%%%%%%%%%%%%%%%%%%%%%%%%%%%%%%%%%%%%%%%%%%%%%%%%%%%%
%{ \tableofcontents}

%%%%%%%%%%%%%%%%%%%%%%%%%%%%%%%%%%%%%%%%%%%%%%%%%%%%%%%%%%%%%%%%%%%%%%%%%%%%%
%================================================================================================================================================
\section{Introduction}
\label{section int} The project Gauss-Manin connection in
disguise  started in the articles \cite{ho22, HosseinMurad} and the book  \cite{GMCD-MQCY3} aims to unify  modular and
automorphic forms with topological string partition  functions  of string
theory, see for instance \cite{alim11} and the references therein.
Modular and automorphic forms have a vast amount of  applications in number
theory and so it is highly desirable to seek for such applications
for $q$-expansions of Physics. The main ingredient of this unification is a natural generalization
of Ramanujan relations between the Eisenstein series interpreted as a vector field in a
moduli space of enhanced elliptic curves. This has been extensively used in transcendental number
theory, see \cite{nes01, Zudilin2011} for an overview of some results.
The starting point is either a Picard-Fuchs equation or a family of
algebraic varieties. In \cite{GMCD-MQCY3} the first author has described the construction of
such vector fields attached to Calabi-Yau equations of the list
in \cite{alenstzu}, and in particular the well-known 14 family of Calabi-Yau threefolds in \cite{dormor}. In this article
we are going to consider the family of $n$-dimensional Calabi-Yau
varieties $X=X_\psi,\ \psi\in
\Pn 1-\{0,1,\infty\}$ obtained by a quotient and desingularization
of the so-called Dwork family:
\begin{equation}
\label{12jan2016}
x_0^{n+2}+x_1^{n+2}+\cdots+x_{n+1}^{n+2}-(n+2)\psi x_0x_1\cdots x_{n+1}=0,
\end{equation}
where $x_0,x_1,\cdots, x_{n+1}$ are homogeneous coordinates of $\Pn {n+1}$. From now on we call $X_\psi$ a \emph{mirror (Calabi-Yau)
variety}, as for $n=1,2,3,4$ it is mirror to generic cubic, quartic, quintic and sextic hypersurfaces in $\Pn {n+1}$ and this is fairely explained in the literature, see \cite{dij95} for $n=1$, \cite{Dolgachev96} for $n=2$, \cite{can91} for $n=3$, \cite{KlemmPandharipande2008} for $n=4$ and \cite{GMPR} for a discussion of an arbitrary $n$.
%This family and its periods are also the main object of study in some physics articles like \cite{GMPR}.
In the present article we
discuss the mentioned  unification in the case of Dwork family, namely we
explain a construction of a modular vector field $\Ra_n=\Ra$
attached to $X_\psi$ such that for $n=1,2$ it has solutions in terms
of (quasi)-modular forms, for $n=3$ the topological partition
functions are rational functions in the coordinates of a solution
of $\Ra$, and for $n\geq 4$ one gets $q$-expansions beyond the
so-far well-known special functions. It is worth pointing out that
we can consider the modular vector field $\Ra$ as an extension of
the systems of differential equations introduced by  G. Darboux
\cite{da78}, G. H. Halphen \cite{ha81} and S. Ramanujan \cite{ra16},
for more details see \cite{ho14},  \cite[\S~1]{nikdelan14}.

For the purpose of Introduction, we need only to know that for
a mirror variety $X$ associated to the Dwork family, $\dim H^n_\dR(X)=n+1$, where $H_\dR^n(X)$ is
the $n$-th algebraic de Rham cohomology of $X$, and its Hodge
numbers $h^{i,j},\ i+j=n,$ are all one. For $n=3$ this is also called
the family of mirror quintic. Let $\T=\T_n$ be the moduli of pairs
$(X,[\alpha_1,\cdots ,\alpha_n,\alpha_{n+1}])$, where
$$
\alpha_i\in F^{n+1-i}\setminus F^{n+2-i},\ \ i=1,\cdots,n,n+1,
$$
$$
[\langle \alpha_i,\alpha_j\rangle]=\imc_n.
$$
Here $F^i$ is the $i$-th piece of the Hodge filtration of
$H^n_\dR(X)$, $\langle\cdot,\cdot \rangle$ is the intersection form
in $H^n_\dR(X)$ and $\imc=\imc_n$ is the explicit constant matrix
given by
\begin{equation}\label{eq phi odd}
\Phi_n:=\left( {\begin{array}{*{20}c}
   {0_{\frac{n+1}{2}} } & {J_{\frac{n+1}{2}}}   \\
   { - J_{\frac{n+1}{2}}}  & {0_{\frac{n+1}{2}}}   \\
\end{array}} \right),
\end{equation}
if $n$ is an odd positive integer, and
\begin{equation}\label{eq phi even}
\Phi_n:=J_{n+1},
\end{equation}
if $n$ is an even positive integer, where by $0_{k}, k\in
\mathbb{N},$ we mean a $k\times k$ block of zeros, and $J_k$ is the
following $k\times k$ block
\begin{equation}\small
J_{k }  := \left( {\begin{array}{*{20}c}
   0 & 0 &  \ldots  & 0 & 1  \\
   0 & 0 &  \ldots  & 1 & 0  \\
    \vdots  &  \vdots  &  {\mathinner{\mkern2mu\raise1pt\hbox{.}\mkern2mu
 \raise4pt\hbox{.}\mkern2mu\raise7pt\hbox{.}\mkern1mu}}  &  \vdots  &  \vdots   \\
   0 & 1 &  \ldots  & 0 & 0  \\
   1 & 0 &  \ldots  & 0 & 0  \\
\end{array}} \right),
\end{equation}
with $J_1=1$.
We construct the universal family $\X\to\T$ together with global
sections $\alpha_i,\ \ i=1,\cdots, n+1$ of the relative algebraic de
Rham cohomology $H^n_\dR(\X/\T)$.
%We further construct a partial compactification $\bar\T\cong \C^{d}$ of $\T$ with a coordinate
%system $(t_0,t_1,\ldots,t_{d-1})$ such that the complement of $\T$ in $\bar \T$ is given by
%$$
%t_{n+1}(t_{n+1}-t_0^{n+2})\prod_{i=1}^st_i=0.
%$$
Let
$$
\nabla:H_{\dR}^{n}(\X/\T)\to \Omega_\T^1\otimes_{\O_\T}H_{\dR}^{n}(\X/\T),
$$
be the algebraic Gauss-Manin connection on $H^n_\dR(\X/\T)$. Our main theorem is:
\begin{theo}
 \label{main3}
There exist a unique vector field $\Ra:=\Ra_n$ and regular functions $\Yuk_i, \ i=1,2,\ldots,n-2$ in $\T$
such that the Gauss-Manin connection of the universal family of $n$-fold
mirror variety $\X/\T$ composed with the vector field $\Ra$,
namely $\nabla_{\Ra}$, satisfies:
\begin{equation}
\label{jimbryan}
\nabla_{\Ra}
\begin{pmatrix}
\alpha_1\\
\alpha_2 \\
\alpha_3 \\
\vdots \\
\alpha_n \\
\alpha_{n+1} \\
\end{pmatrix}
= \underbrace {\begin{pmatrix}
0& 1 & 0&0&\cdots &0&0\\
0&0& \Yuk_1&0&\cdots   &0&0\\
0&0&0& \Yuk_2&\cdots   &0&0\\
\vdots&\vdots&\vdots&\vdots&\ddots   &\vdots&\vdots\\
0&0&0&0&\cdots   &\Yuk_{n-2}&0\\
0&0&0&0&\cdots   &0&-1\\
0&0&0&0&\cdots   &0&0\\
\end{pmatrix}}_\Yuk
\begin{pmatrix}
\alpha_1\\
\alpha_2 \\
\alpha_3 \\
\vdots \\
\alpha_n \\
\alpha_{n+1} \\
\end{pmatrix},
\end{equation}
and  $\Yuk\Phi+\Phi \Yuk^\tr=0$. In fact,
\begin{equation}
 \label{thanksdeligne}
 \T:=\spec(\C[t_1,t_2,\ldots, t_{\dt},\frac{1}{t_{n+2}(t_{n+2}-t_1^{n+2})\check t %\prod_{i=1}^s t_{j_i}
 }]),
\end{equation}
where
\begin{equation}
\label{29dec2015} \dt=\dt_n=\left \{
\begin{array}{l}
\frac{(n+1)(n+3)}{4}+1,\,\, \quad  \textrm{\rm if \textit{n} is odd}
\\\\
\frac{n(n+2)}{4}+1,\,\,\,\,\quad\quad \textrm{\rm if \textit{n} is
even}
\end{array} \right. ,
\end{equation}
and $\check t$ is a product of $s$ variables among $t_i$'s,
$i=1,2,\ldots,\dt, \ i\not=1,n+2$ and  $s=\frac{n-1}{2}$ if $n$ is
an odd integer and $s=\frac{n-2}{2}$ if $n$ is an even integer.
%and $t_{j_i}\in \{t_2,t_3,\ldots,t_d\}$.
%The vector field $\Ra_0$ is holomorphic along the divisors
%$t_{n+2}=0$ and $(t_{n+2}-t_1^{n+2})$ and has a pole of order one
%along the other divisors $t_i=0, \ i=1,2,\ldots,s$.
\end{theo}
In the proof of Theorem \ref{main3} we will show more than what we
announce  in its statement.  Indeed, we will give the
regular functions $\Yuk_i$'s explicitly, and we will find an
algorithm to express the modular vector field $\Ra$.  An explicit
expression for $\Ra_{3}$ has been given in \cite{ho22, GMCD-MQCY3}
by the first author. In the next theorem we find $\Ra_{1}$ and
$\Ra_{2}$ explicitly and express their solutions in terms of
quasi-modular forms. For a similar computation of $\Ra_1$ and $\Ra_2$ in the context of special geometry see \cite{Murad2014}.
The modular structure for $n=1,2$, that is mirror cubic and quartic, has been discussed respectively in \cite{Ceresole:review} and
\cite{LiYa96}, and also in the GMCD context in \cite{Murad2014}. These references include discussions of the derivation of the modular
group from the monodromies of the periods of the mirror variety.
Modular properties of $\Ra_n$ for $n\geq 3$ is not known and after the present text was ready,  the second author in
\cite{younes2} was able to find $\mathfrak{sl}_2(\C)$
Lie algebras involving the vector field $\Ra_n$.
\begin{theo}
For $n=1,2$ the vector field $\Ra$ as an ordinary differential
equation is respectively given by
\begin{equation}
 \label{lovely-1}
 \Ra_{1}:  \left \{ \begin{array}{l} \dot
t_1=-t_1t_2-9(t_1^3-t_3)
\\ \\
\dot t_2=81t_1(t_1^3-t_3)-t_2^2
\\ \\
\dot t_3=-3t_2t_3
\end{array} \right. ,
\end{equation}
where $\dot \ast=3\cdot q\cdot \frac{\partial \ast}{\partial q}$,
and
\begin{equation}
 \label{lovely-2}
 \Ra_{2}:\left \{ \begin{array}{l}
\dot{t}_1=t_3-t_1t_2\\ \\
\dot{t}_2=2t_1^2-\frac{1}{2}t_2^2\\ \\
\dot{t}_3=-2t_2t_3+8t_1^3\\ \\
\dot{t}_4=-4t_2t_4
\end{array} \right.,
\end{equation}
where $\dot \ast=-\frac{1}{5}\cdot q\cdot \frac{\partial
\ast}{\partial q}$, and the following polynomial equation holds
among $t_i$'s
\begin{equation}\label{eq t3}
 t_3^2=4(t_1^4-t_4).
\end{equation}
\end{theo}
In the above theorem $q$ is a free parameter.  For a complex number $\tau$ with $\rm\, Im \tau>0$, if
we set $q=e^{2\pi i \tau}$, then we find the following solutions of
$\Ra_{1}$ and $\Ra_{2}$ respectively:
\begin{equation} \label{eq solution R01}
\left \{ \begin{array}{l}
{t}_1(q)=\frac{1}{3}(2\theta_3(q^2)\theta_3(q^6)-\theta_3(-q^2)\theta_3(-q^6)),\\ \\
{t}_2(q)=\frac{1}{8}(E_2(q^2)-9E_2(q^6)),\\ \\
{t}_3(q)=\frac{\eta^9({q}^3)}{\eta^3({q})},
\end{array} \right.
\end{equation}
and
\begin{equation} \label{eq solution R02}
\left \{ \begin{array}{l}
\frac{10}{6}{t}_1(\frac{q}{10})=\frac{1}{24}(\theta_3^4(q^2)+\theta_2^4(q^2)),\\ \\
\frac{10}{4}{t}_2(\frac{q}{10})=\frac{1}{24}(E_2(q^2)+2E_2(q^4)),\\ \\
10^4{t}_4(\frac{{q}}{10})=\eta^8({q})\eta^8({q}^2),
\end{array} \right.
\end{equation}
where  $E_2$, $\eta$ and  $\theta_i$'s  are the classical
Eisenstein, eta and theta series, respectively, given as follows:
\begin{align}
&E_2(q)=1-24\sum_{k=1}^\infty \sigma(k) q^{k}\,\, with\,\,
\sigma(k)=\sum_{d\mid k}d,\\
&\eta({q})={q}^{\frac{1}{24}}\prod_{k=1}^\infty (1-{q}^{k}),\\
&\theta_2(q)=\sum_{k=-\infty}^\infty
q^{\frac{1}{2}(\frac{k+1}{2})^2}\,\, , \,\,\,
\theta_3(q)=1+2\sum_{k=1}^\infty q^{\frac{1}{2}k^2}.
\end{align}
We have checked this statement for the first $100$ coefficients of q-expansions,
and the proof  can be done in a similar way as of
Ramanujan's or Darboux's case.
The first 16 coefficients of  \eqref{eq solution R01} and \eqref{eq solution R02} are
listed in  Table \ref{table1}.
%\begin{rem}
%We recall that $\eta$ and  $\theta_i$'s are modular forms, and $E_2$
%is a quasi-modular form.
%\end{rem}
We study  the coefficients of q-expansions of the solutions
given in (\ref{eq solution R01}) and (\ref{eq solution R02}), and we
find some interesting enumerative properties. For example, in
(\ref{eq solution R01}) the coefficients of
$t_1(q)=\sum_{k=0}^\infty t_{1,k}q^k$ have the following enumerative
property. Let $k$ be a non-negative integer. If $k=4m,\ \ m\in \N$,
then the equation $x^2+3y^2=k$ has $3t_{1,k}$ integer solutions.
Otherwise the equation has $t_{1,k}$ integer solutions. For more
properties of this type see Section \ref{section ep}.

The article is organized in the following way. In Section
\ref{Dwork-section} we review and summarize some basic facts about the structure of the
Dwork family from which the
mirror variety $X_\psi$ arises. In Section \ref{section ms1} we
introduce the notion of moduli space of holomorphic $n$-form $\Ts$,
and we see that $\Ts$ is two dimensional and present a coordinate system
for it. Section \ref{intersection} deals with the calculation
of intersection form matrix of a given basis of the de Rham
cohomology of mirror variety. In Section \ref{section EMS} we
present the moduli space $\T$ and construct a complete coordinate
system for $\T$. Section \ref{section gmc} is devoted to the
computing of Gauss-Manin connection of the families $\X/\Ts$ and
$\X/\T$. In Section \ref{section main3} Theorem \ref{main3} is
proved and the modular vector field is explicitly computed  for
$n=1,2,3,4$. Finally,  in Section \ref{section ep} after finding the
solutions of $\Ra_{1}$ and $\Ra_{2}$ in terms of quasi-modular
forms, we proceed with the study of enumerative properties of the
$q$-expansions of the solutions.

{\bf Acknowledgment.} The second author thanks the "Instituto de
Matem\'atica da Universidade Federal do Rio de Janeiro" of Brazil,
where he did a part of this work during his Postdoctoral research
with the grant of "CAPES", and in particular he would like to thank
Bruno C. A. Scardua for his support. We are very grateful to the referees
whose critical comments improved the present article.

%================================================================================================================================================
\section{Dwork Family}\label{Dwork-section}
Let $f_\psi$ be the polynomial in the left hand side of \eqref{12jan2016}.
 Let also $W_\psi$ be the
$n$-dimensional hypersurface in $\P^{n+1}$ given by $f_\psi$. We
know that the first Chern class of $W_\psi$ is zero, from which
follows that $W_\psi$ is a \cy manifold. Thus we have a family of
\cy manifolds given by $\pi:\W \to \P^1$, where $\W\subset
\P^{n+1}\times \C$, $W_\psi=\pi^{-1}(\psi)$ and
\[
W_\infty=\{(x_0,x_1,\ldots,x_{n+1})\mid \,\, x_0x_1\ldots
x_{n+1}=0\}.
\]
This family has been used by B. Dwork  in order to develop  the  deformation
theory of  zeta functions of  nonsingular hypersurfaces in a
projective space, see \cite{dw1,dw2}. One can easily see that the
singular points of this family are $\psi^{n+2}=1,\infty$. Let $G$ be
the following group
\begin{equation}\label{eq group 1}
G:=\{(\zeta_0,\zeta_1,\ldots,\zeta_{n+1})\mid \,\, \zeta_i^{n+2}=1,
\ \zeta_0\zeta_1\ldots \zeta_{n+1}=1 \},
\end{equation}
which acts on $W_\psi$ as follow
\begin{equation} \label{eq group action 1}
(\zeta_0,\zeta_1,\ldots,\zeta_{n+1}).(x_0,x_1,\ldots,x_{n+1})=(\zeta_0x_0,\zeta_1x_1,\ldots,\zeta_{n+1}x_{n+1}).
\end{equation}
We denote by
$Y_\psi:=W_\psi/G$ the quotient space of this action, which is quite
singular. Indeed, $Y_\psi$ is singular in any $x\in W_\psi$ that its
stabilizer in $G$ is nontrivial. For $\psi^{n+2} \neq 1,\infty$
there exist a resolution $X_\psi\to Y_\psi$ of singularities of
$Y_\psi$, such that $X_\psi$ is a \cy $n$-fold with
$\hn^{i,j}(X_\psi)=1,\,\, i+j=n$. Therefore, we have a new family
where the fibers are \cy $n$-folds $X_\psi$ which is the mirror
family of $W_\psi$, see \cite{GMPR}. The de Rham cohomology $H^n_\dR(X_{\psi})$ can be identified in a natural way with  the equivariant cohomology $H^n_\dR(W_{\psi})_G$, and in practice one uses this, and the knowledge of resolution of singularities will not be used throughout the paper.
 The standard variable which
is used in the literature is defined by
$$
z:=\psi^{-(n+2)}.
$$
We have the families $W_z$ given by $f_z=0$, where
\[
f_z(x_0,x_1,\ldots,x_{n+1}):=zx_0^{n+2}+x_1^{n+2}+x_2^{n+2}+\cdots+x_{n+1}^{n+2}-(n+2)x_0
x_1x_2\cdots x_{n+1},
\]
and $X_z$ is defined similar to $X_\psi$.  The new set of singularities is given by $z=0,1 \,\ \textrm{and} \,\
\infty$. From now on we call $X_z$ (or $X_\psi$) the \emph{mirror
variety}. There is a global holomorphic $(n,0)$-form $\eta \in
H^n_\dR(X_z)$ which is given by
\[
\eta:=\frac{dx_1\wedge dx_2\wedge \ldots \wedge dx_{n+1}}{df_z}.
\]
in the affine chart $\{x_0=1\}$. The periods $\int_{\delta}\eta,\ \ \delta\in H_n(X_z,\Z)$ satisfy the well-known Picard-Fuchs equation
\begin{eqnarray}
 & & \L\left(\int_{\delta}\eta\right) =  0 ,\, \,  \hbox{ where}\\
\label{eq pf eta}
& & \L := \vartheta^{n+1}-z(\vartheta+\frac{1}{n+2})(\vartheta+\frac{2}{n+2})\ldots(\vartheta+\frac{n+1}{n+2}), \ \ \ \
\vartheta:=z\frac{\partial}{\partial z}
\end{eqnarray}
Note that if $n=1,2$ or $3$ respectively,  then $X_z$ is
a special family of  elliptic curves, $K3$-surfaces and  mirror quintic $3$-folds, respectively. Note that $X_z$ for $n=1,2$ is not the generic elliptic curve nor
the generic K3 but rather the cubic curve and the quartic K3 surface.

%==================================================================================================================
\section{Moduli Space of holomorphic $n$-forms}\label{section ms1}

We denote by  $\Ts$  the moduli of pairs
$(X,\alpha)$, where $X$ is an $n$-dimensional mirror variety and
$\alpha$ is a holomorphic $n$-form on $X$. We know that the family
$X_z$ is a one parameter family and the $n$-form $\alpha$ is unique,
up to multiplication by a constant, therefore $\dim\, \Ts=2$. The multiplicative group
$\Gm:=(\C^*,\cdot)$ acts on $\Ts$ by:
 $$
(X,\alpha)\bullet k=(X,k^{-1}\alpha),\ k\in \Gm,\ (X,\alpha)\in \Ts.
$$
We present a chart $(t_1,t_{n+2})$ for $\Ts$. To do
this, for any $(t_1,t_{n+2})\in \C^2$ we define the following
polynomial
\[
f_{t_1,t_{n+2}}(x_0,x_1,\ldots,x_{n+1}):=t_{n+2}x_0^{n+2}+x_1^{n+2}+x_2^{n+2}+\cdots+x_{n+1}^{n+2}-(n+2)t_1x_0
x_1x_2\cdots x_{n+1}.
\]
Note that $t_{n+2}$ is only multiplied with the monomial $x_0^{n+2}$ and not all the monomials in the expression except the last one. This will be essential in the proof of Proposition \ref{14sept2019}. The choice of the monomial $x_0^{n+2}$ is not relevant. In fact, if we choose to define $f_{t_1,t_{n+2}}$ with $t_{n+2}x_1^{n+2}$ then the two families of varieties are isomorphic under the linear transformation $[x_0:x_1:x_2:\cdots:x_{n+1}] \mapsto
[t_{n+2}^{\frac{1}{n+2}}x_0: t_{n+2}^{-\frac{1}{n+2}}x_1: x_2:\cdots:x_{n+1}]$.
The discriminant of $f_{t_1,t_{n+2}}$ is given by
$\Delta_{t_1,t_{n+2}}=(t_{n+2}-t_1^{n+2})t_{n+2}$. Let
${\textsf{W}}_{t_1,t_{n+2}}$ be the following two parameter family
of \cy manifolds
\[
{\textsf{W}}_{t_1,t_{n+2}}:=\{(x_0,x_1,\ldots,x_{n+1})\, |\,
f_{t_1,t_{n+2}}(x_0,x_1,\ldots,x_{n+1})=0\}\subset \P^{n+1}.
\]
${\sf W}_{t_1,t_{n+2}}$ is singular if and only if
$\Delta_{t_1,t_{n+2}}=0$. For any
$$
(t_1,t_{n+2})\in \C^2\setminus
\{(t_1,t_{n+2})\, | \,\Delta_{t_1,t_{n+2}}=0\}
$$
we let $\X_{t_1,t_{n+2}}$ to be the resolution of the singularities
of ${\textsf{W}}_{t_1,t_{n+2}}/G$ where the group $G$ and the group
action are given by (\ref{eq group 1}) and (\ref{eq group action
1}). Next we fix the $n$-form $\omega_1$ on the family
$\X_{t_1,{t_{n+2}}}$, where $\omega_1$ in the affine space $\{x_0=1\}$
is given by
\[
\omega_1:=\frac{dx_1\wedge dx_2\wedge \ldots \wedge
dx_{n+1}}{df_{t_1,t_{n+2}}}.
\]
Note that for $(t_1,t_{n+2})=(1,z)$ we have
$(\X_{t_1,{t_{n+2}}},\omega_1)=(X_z,\eta)$.
\begin{prop}
	\label{14sept2019}
We have
$$
\Ts=\spec \left( \C\left[t_1,t_{n+2},\frac{1}{(t_1^{n+2}-t_{n+2})t_{n+2}}\right]\right )
$$
and the morphism $\X\to\Ts$ is the the universal family of
$(X,\alpha)$,
where $X$ is an $n$-dimensional mirror variety and
$\alpha$ is a holomorphic $n$-form on $X$. Moreover, the $\Gm$-action on $\Ts$ is given by
\begin{equation}
\label{poloar}
(t_1,t_{n+2})\bullet k=(kt_1,k^{n+2}t_{n+2}),\ (t_1,t_{n+2})\in \Ts,\ k\in \Gm.
\end{equation}
\end{prop}
\begin{proof}
We have a map which sends a point  $(t_1,t_{n+2})\in\Ts$ to the
pair $(\X_{t_1,t_{n+2}},\omega_1)$ in the moduli space $\Ts$ as a
set. Its inverse is given by
$$
(X_{z},a\eta)\mapsto
   (a^{-1},za^{-(n+2)}),\ \ \ a\in\Gm.\nonumber
$$
Note that  $(\X_{t_1,t_{n+2}},\omega_1)$ and $(X_z,t_1^{-1}\eta)$,
where $z=\frac{t_{n+2}}{t_1^{n+2}}$, in the moduli space $\Ts$
represent the same element. The affirmation concerning the
$\Gm$-action follows from the isomorphism:
\begin{align}
 &(\X_{kt_1,k^{n+2}t_{n+2}},\ k\omega_1)\cong (\X_{t_1,t_{n+2}}, \omega_1), \label{20m2014}\\
 &(x_1,x_2,\cdots ,x_{n+1})\mapsto
 (k^{-1}x_1,k^{-1}x_2,\cdots,k^{-1}x_{n+1}),\nonumber
\end{align}
given in the affine coordinates $x_0=1$.
\end{proof}

%=================================================================================================================
\section{Intersection form and Gauss-Manin connection}\label{intersection}
Let $X$ be an $n$-dimensional mirror variety and $\xi_1,\xi_2\in
H^n_\dR(X)$. Then in the context of de Rham cohomology, the
\emph{intersection form} of $\xi_1$ and $\xi_2$, denoted by $\langle
\xi_1,\xi_2 \rangle$, is given by
$$
\langle \xi_1,\xi_2 \rangle=\frac{1}{(2\pi i)^n}\int_{X}\xi_1\wedge \xi_2.
$$
We recall that $\langle .,. \rangle$ is a non-degenerate
$(-1)^n$-symmetric form, and
\begin{equation}\label{eq intHodge}
\langle F^i,F^j\rangle=0, \ i+j\geq n+1,
\end{equation}
where
$$
F^\bullet: \{0\}=F^{n+1}\subset F^n\subset \ldots \subset F^1\subset
F^0=H^{n}_\dR(X),\ \ \dim\, F^i=n+1-i\ ,
$$
is the Hodge filtration of $H^{n}_\dR(X)$.

Let
\begin{equation} \label{eq gmn}
 \nabla:H_{\dR}^{n}(\X/\Ts)\to \Omega_\Ts^1\otimes_{\O_\Ts}H_{\dR}^{n}(\X/\Ts)
\end{equation}
be the Gauss-Manin connection of the  two parameter family of varieties
$\X/\Ts$, and $\frac{\partial}{\partial t_1}$ be a vector field on
the moduli space $\Ts$. For simplicity, we use the same notation
$\frac{\partial}{\partial t_1}$, to show
$\nabla_{\frac{\partial}{\partial t_1}}$ which is the composition of
the Gauss-Manin connection $\nabla$ with the vector field
$\frac{\partial}{\partial t_1}$.  Now we define new $n$-forms
$\omega_i,\,\ i=1,2,\ldots, n+1$, as follows
\begin{equation}
\label{29oct11} \omega_i:= {\frac{\partial^{i-1}}{\partial
t_1^{i-1}}}(\omega_1).
\end{equation}
Later, in Proposition \ref{lemm 127} we will see that
$\omega_1,\omega_2,\ldots,\omega_{n+1} $ form a basis of
$H^{n}_\dR(X)$ compatible with its Hodge filtration, i.e.
\begin{equation}
\label{eq. Gr. tr.}\omega_i\in
F^{n+1-i}\setminus F^{n+2-i}, i=1,2,\ldots, n+1.
\end{equation}
We write the Gauss-Manin connection of $\X/\Ts$ in the basis
$\omega$ as follow
\begin{equation} \label{eq Atilde}
\nabla\omega=\gmo \omega,
\end{equation}
and we denote by
\begin{equation}
\label{1mar2016} \Omega=\Omega_n:=\left( \langle
\omega_i,\omega_j\rangle\right)_{1\leq i,j\leq n+1},
\end{equation}
the intersection form matrix in the basis $\omega$. We have
\begin{equation}
\label{14/12/2015} d\Omega=\gmo \Omega+ \Omega \gmo^{\tr}.
\end{equation}
The entries of $\gmo$ and $\Omega$ are  respectively regular
differential 1-forms and functions in $\Ts$. For arbitrary $n$, we
do not  have a general formula for $\Omega$ and $\gmo$.  We
have only an algorithm which computes the entries of $\Omega$ and
$\gmo$ recursively.
%Before that, we announce the following
%proposition that gives the Picard-Fuchs equation associated with the
%$n$-form $\omega_1$, and we will prove it later.
%\begin{prop} \label{prop pf omega}
For $n=1,2,3,4$ the Picard-Fuchs equation associated with the $n$-form $\omega_1$ is
given by {\small
\begin{align}
%\label{prop pf omega}
\frac{\partial^{n+1}}{{\partial
t_1}^{n+1}}&=-S_2(n+2,n+1)\frac{t_1^{n+1}}{t_1^{n+2}-t_{n+2}}\frac{\partial^{n}}{{\partial
t_1}^{n}}-S_2(n+2,n)\frac{t_1^{n}}{t_1^{n+2}-t_{n+2}}\frac{\partial^{n-1}}{{\partial
t_1}^{n-1}}-\ldots \label{eq pf omega}\\
&-S_2(n+2,2)\frac{t_1^{2}}{t_1^{n+2}-t_{n+2}}\frac{\partial}{{\partial
t_1}}-S_2(n+2,1)\frac{t_1}{t_1^{n+2}-t_{n+2}}\, ,\nonumber
\end{align}
}
where $S_2(r,s), \, r,s\in \N,$ refers to Stirling number of the
second kind which is given by
\begin{equation}
\label{16jan2016}
S_2\left( {r,s} \right){\rm{ }} = {\rm{ }}\frac{1}{{s!}}{\rm{ }}\sum\limits_{i = 0}^s {{{( - 1)}^i}
\left( {\begin{array}{*{20}{c}}
s\\
i
\end{array}} \right)} {\left( {s - i} \right)^r}.
\end{equation}
For details of the computation in the mirror quintic case ($n=3$) see \cite[\S 3.8]{GMCD-MQCY3}.
The equation \eqref{eq pf omega} must be true for arbitrary $n$, however, we are only interested to compute this for explicit $n$'s and so
we do not provide a proof for arbitrary $n$.

\begin{prop} \label{lemm 127}
We have
\begin{description}
  \item[(i)] $\langle \omega_i,\omega_j \rangle =0$, if $i+j\leq n+1$.
  \item[(ii)] $\langle \omega_1,\omega_{n+1} \rangle
  =(-(n+2))^n\frac{c_n}{t_1^{n+2}-t_{n+2}}$, where $c_n$ is a constant.
  \item[(iii)]$\langle \omega_j,\omega_{n+2-j} \rangle =(-1)^{j-1}\langle \omega_1,\omega_{n+1} \rangle$, for $j=1,2,\ldots,n+1$.
  \item[(iv)]
  We can determine all the rest of  $\langle\omega_i,\omega_j\rangle$'s in a unique way.
\end{description}
\end{prop}
\begin{proof}
Note that the intersection form is well-defined for all points in $\Ts$, and so,
$\langle \omega_i,\omega_j\rangle$'s are regular functions in $\Ts$. This implies that  they have
poles only along $t_{n+2}=0$ and $t_{n+2}-t_1^{n+2}=0$.
\begin{description}
  \item[(i)] The Griffiths transversality implies that
$$
\omega_i\in F^{n+1-i}, i=1,2,\ldots, n+1.
$$
This property and the property given in (\ref{eq intHodge}) complete
the proof of {\bf (i)}.
  \item[(ii)] If we present the Picard-Fuchs equation associated with the  holomorphic $n$-form
  $\eta$ as follow:
  \begin{equation}
  \vartheta^{n+1}=a_0(z)+a_1(z)\vartheta+\ldots+a_n(z)\vartheta^n,
  \end{equation}
  then because of (\ref{eq pf eta}) we find
  $$a_n(z)=\frac{n+1}{2}\frac{z}{1-z}\ .$$
  One can verify the differential equation given below
  \[
    \vartheta \langle \eta, \vartheta^n \eta \rangle +
    \frac{2}{n+1}a_n(z)\langle \eta, \vartheta^n \eta \rangle=0,
  \]
  from which we get $\langle \eta, \vartheta^n \eta \rangle= c_n\exp \left( {-\frac{2}{n+1}
\int_0^z a_n(v) \frac{dv}{v}}\right)$, where $c_n$ is a constant.
This yields
\begin{equation}\label{eq at}
\langle \eta, \vartheta^n \eta \rangle=\frac{c_n}{1-z}.
\end{equation}
On the other hand in Section \ref{section ms1} we saw
$z=\frac{t_{n+2}}{t_1^{n+2}}$, which implies
$\vartheta=z\frac{\partial}{\partial
z}=-\frac{1}{n+2}t_1\frac{\partial}{\partial t_1}$. One can easily
see that $\eta=t_1\omega_1$, hence
\begin{align}
\vartheta^n \eta&=(-\frac{1}{n+2}t_1\frac{\partial}{\partial
t_1})^n(t_1\omega_1 )\nonumber\\
&=b_1\omega_1+\ldots+b_n \omega_{n}+(-\frac{1}{n+2})^nt_1^{n+1}\omega_{n+1},
\nonumber
\end{align}
where $b_j$'s are rational functions in $t_1,t_{n+1}$. Therefore, {\bf (i)} implies
\[
\langle \eta, \vartheta^n \eta \rangle=\langle
t_1\omega_1,(-\frac{1}{n+2})^nt_1^{n+1}\omega_{n+1} \rangle\ ,
\]
which completes the proof of {\bf (ii)}.
  \item[(iii)] By {\bf (i)} we have $\langle \omega_{j},\omega_{n+1-j} \rangle=0,\,\, j=1,2,\ldots,n$.
  Thus we get
\begin{align}
\frac{\partial}{\partial t_1} \langle \omega_{j},\omega_{n+1-j}
\rangle&=\langle \frac{\partial}{\partial t_1}
\omega_{j},\omega_{n+1-j} \rangle+\langle
\omega_{j},\frac{\partial}{\partial t_1} \omega_{n+1-j} \rangle
\nonumber\\ &=\langle \omega_{j+1},\omega_{n+1-j} \rangle+\langle
\omega_{j},\omega_{n+2-j} \rangle=0,\nonumber
\end{align}
hence we obtain $\langle \omega_{j+1},\omega_{n+1-j}
\rangle=-\langle \omega_{j},\omega_{n+2-j} \rangle\ , \ \
j=1,2,\ldots,n$, from which follows {\bf (iii)}.
  \item[(iv)] We present the desired algorithm. So far, we have
  computed the first row of the matrix $\Omega$. Suppose that we
  have the $i$-th row of $\Omega$, $1\leq i\leq n$, and then determine $(i+1)$-th row. To
  compute $\langle \omega_{i+1},\omega_j \rangle, \, n+2-i\leq j \leq
  n+1$, we apply $\frac{\partial}{\partial t_1}\langle \omega_{i},\omega_j
  \rangle$, which implies
  \[
    \langle \omega_{i+1},\omega_{j}
  \rangle =\frac{\partial}{\partial t_1}\langle \omega_{i},\omega_j
  \rangle - \langle \omega_{i},\omega_{j+1} \rangle\, .
  \]
  Note that if $j=n+1$, then $\omega_{n+2}=\frac{\partial^{n+1}}{\partial
  t_1^{n+1}}(\omega_1)$ and we compute it by using the Picard-Fuchs
  equation given in \eqref{eq pf omega}.
\end{description}
\end{proof}
%The proof follows the same line as in \cite[Proposition
%3.2~]{younes1}. The Griffiths transversality implies that
%$$
%\omega_i\in F^{n+1-i}, i=1,2,\ldots, n+1
%$$.
%$$
%\theta \imF=A\imF+\imF A^{\tr}.
%$$
%$$
%\theta B=A\cdot B,
%$$
%where
% \begin{equation}\label{17july2012}
% A=\begin{pmatrix}
%    0&1&0&0\\
%    0&0&1&0\\
%    0&0&0&1\\
%    a_0(z)& a_1(z)& a_2(z)&a_3(z)
%   \end{pmatrix},\ \ \
% B:= \begin{pmatrix}
%      \eta \\
%      \theta\eta\\
%       \theta^2\eta\\
%       \theta^3\eta
%     \end{pmatrix}.
% \end{equation}

The intersection form matrix for $n=1,2,3,4$ are respectively given as
follows:
\begin{equation}\small
\Omega_1=\left({\large
  \begin{array}{cc}
     0 & -\frac{3c_1}{t_1^3-t_3} \\
       \frac{3c_1}{t_1^3-t_3} & 0 \\
  \end{array}}
\right)\,\ , \,\,\,\ \Omega_2=\left({\large
  \begin{array}{ccc}
    0 & 0 & \frac{16c_2}{t_1^4-t_4} \\
    0 & -\frac{16c_2}{t_1^4-t_4} & \frac{32c_2t_1^3}{(t_1^4-t_4)^2}
    \\
    \frac{16c_2}{t_1^4-t_4} & \frac{32c_2t_1^3}{(t_1^4-t_4)^2} & \frac{-16c_2t_1^2(5t_1^4-t_4)}{(t_1^4-t_4)^3} \\
  \end{array}}
\right)\,\ , \nonumber
\end{equation}
\[
 \Omega_3=
\begin{pmatrix}
0             &       0     &      0    &       -\frac{1}{625(t_1^5-t_5)} \\
0             &       0     &    \frac{1}{625(t_1^5-t_5)}  & -\frac{t_1^4}{625(t_1^5-t_5)^2} \\
0             &       -\frac{1}{625(t_1^5-t_5)}&0&           \frac{t_1^3}{625(t_1^5-t_5)^2} \\
\frac{1}{625(t_1^5-t_5)} & \frac{t_1^4}{625(t_1^5-t_5)^2} & -\frac{t_1^3}{625(t_1^5-t_5)^2} &0
\end{pmatrix},
\]
\begin{equation}\small \Omega_4=\left({\large
  \begin{array}{ccccc}
    0 & 0 & 0 & 0 & \frac{6^4c_4}{t_1^6-t_6} \\
    0 & 0 & 0 & -\frac{6^4c_4}{t_1^6-t_6} & \frac{9\times6^4c_4t_1^5}{(t_1^6-t_6)^2}
    \\
    0 & 0 & \frac{6^4c_4}{t_1^6-t_6} & -\frac{3\times6^4c_4t_1^5}{(t_1^6-t_6)^2} & \frac{6^4c_4t_1^4(7t_1^6+20t_6)}{(t_1^6-t_6)^3}
    \\
    0 & -\frac{6^4c_4}{t_1^6-t_6} & -\frac{3\times6^4c_4t_1^5}{(t_1^6-t_6)^2} & \frac{6^4c_4t_1^4(14t_1^6-5t_6)}{(t_1^6-t_6)^3} &
    \frac{-6^4c_4t_1^3(56t_1^{12}+35t_1^6t_6-10t_6^2)}{(t_1^6-t_6)^4} \\
     \frac{6^4c_4}{t_1^6-t_6}& \frac{9\times6^4c_4t_1^5}{(t_1^6-t_6)^2} & \frac{6^4c_4t_1^4(7t_1^6+20t_6)}{(t_1^6-t_6)^3} & \frac{-6^4c_4t_1^3(56t_1^{12}+35t_1^6t_6-10t_6^2)}{(t_1^6-t_6)^4} &
    \frac{6^4c_4t_1^2(273t_1^{18}+238t_1^{12}t_6+217t_1^6t_6^2+t_6^3)}{(t_1^6-t_6)^5} \\
  \end{array}}
\right). \nonumber
\end{equation}

\section{Moduli space of enhanced Calabi-Yau varieties} \label{section EMS}
Let  $\T=\T_n$ be  the moduli of  pairs
$(X,[\alpha_1,\alpha_2,\ldots,\alpha_{n+1}])$, where $X$ is an
$n$-fold mirror variety and
$\{\alpha_1,\alpha_2,\ldots,\alpha_{n+1}\}$ is a  basis of
$H^n_\dR(X)$ compatible with its Hodge filtration, and such that the
intersection form matrix of this basis is constant, that is,
\begin{equation}
\label{11jan2016}
 \left(
                           \begin{array}{c}
                             \langle\alpha_i,\alpha_j\rangle \\
                           \end{array}
                         \right)_{1\leq i,j\leq n+1}=\Phi.
\end{equation}
%$\left(
%                           \begin{array}{c}
%                             \langle\alpha_i,\alpha_j\rangle \\
%                           \end{array}
%                         \right)_{1\leq i,j\leq n+1}=\Phi$.
%
If we denote by $\dt_n:=\dim \T_n$, then from \cite[Theorem
1~]{nikdelan14} we get \eqref{29dec2015}.
The objective of this section is to construct a coordinate system
for $\T$.

In Section \ref{intersection} we have fixed the basis $\{\omega_1,\omega_2,\ldots,\omega_{n+1}\}$ of $H^n_\dR(X)$ that is
compatible with its Hodge filtration.  Let $S=\left(
                           \begin{array}{c}
                             s_{ij} \\
                           \end{array}
                         \right)_{1\leq i,j\leq n+1}$ be a lower triangular matrix, whose  entries are indeterminates $s_{ij},\ \ i\geq j$ and $s_{11}=1$.
We define
$$
\alpha:=S\omega,
$$
where
\[
%\alpha=\left( {\begin{array}{*{20}c}
%   {\alpha _1 } & {\alpha _2 } &  \ldots  & {\alpha _{n + 1} }  \\
%\end{array}} \right)^\tr \qquad \& \qquad
\omega:=
\left(
{\begin{array}{*{20}c}
 {\omega _1 } & {\omega _2 } &  \ldots  & {\omega _{n + 1} }  \\
\end{array}} \right)^\tr.
\]
%Then the compatibility of these bases with the Hodge filtration of
%$H^n_\dR(X)$ implies that $S$ is a lower triangular matrix, which we
%consider it as follow
%\begin{equation}\label{eq s n}\small
%S=\left(
%         \begin{array}{ccccc}
%           1 & 0 & 0 & \ldots & 0 \\
%           s_{21} & s_{22}& 0 & \ldots & 0 \\
%           s_{31} & s_{32} & s_{33} & \ldots & 0 \\
%           \vdots & \vdots & \vdots & \ddots & \vdots  \\
%           s_{(n+1)1} & s_{(n+1)2} & s_{(n+1)3} & \ldots & s_{(n+1)(n+1)}  \\
%              \end{array}
%       \right).
%\end{equation}
%where $s_{ij}$'s are unknown parameters. Evidently $s_{ii}\neq 0,
%i=2,3,\ldots, n+1$. One can easily verify that the equality
%$\alpha=S\omega$ implies
%$$
%\left(
%                           \begin{array}{c}
%                             \langle\alpha_i,\alpha_j\rangle \\
%                           \end{array}
%                         \right)_{1\leq i,j\leq n+1}=S\left(
%                           \begin{array}{c}
%                             \langle\omega_i,\omega_j\rangle \\
%                           \end{array}
%                         \right)_{1\leq i,j\leq n+1}S^\tr.
%$$
We assume that $\left(
                           \begin{array}{c}
                             \langle\alpha_i,\alpha_j\rangle \\
                           \end{array}
                         \right)_{1\leq i,j\leq n+1}=\Phi$, and so,  we get the following equation
\begin{equation}\label{eq sost}
S\Omega S^\tr=\Phi.
\end{equation}
If we set  $\Psi=\left(
                           \begin{array}{c}
                             \Psi_{ij} \\
                           \end{array}
                         \right)_{1\leq i,j\leq n+1}:=S\Omega S^\tr$, then $\Psi$ is a $(-1)^n$-symmetric matrix
and $\Psi_{ij}=0$ for $i=1,2,\ldots,n$ and $j\leq n+1-i$. Moreover,
in the case that $n$ is an odd integer we get $\Psi_{ii}=0,\,
i=1,2,\ldots,n+1$.
Therefore, the equation (\ref{eq sost}) gives us
$d_0:=\frac{(n+2)(n+1)}{2}-\dt-2$ equations, where $\dt$ is given by
\eqref{29dec2015}.
%\[
%d_0:=\left \{
%\begin{array}{l}
%\frac{(n+1)^2}{4};\,\, \quad  \textrm{\rm if \textit{n} is odd}
%\\\\
%\frac{(n+2)^2}{4};\,\, \quad \textrm{\rm if \textit{n} is even}
%\end{array} \right..
%\]
These equations are independent from
each other and so we can express $d_0$ numbers of parameters
$s_{ij}$'s in terms of other $\dt-2$ parameters that we fix them as
\emph{independent parameters}. For simplicity we write the first
class of parameters as $\check t_1,\check t_2,\cdots, \check
t_{d_0}$ and the second class as $t_2,
t_3,\ldots,t_{n+1},t_{n+3},\ldots,t_\dt$. We put the independent
parameters $t_i$ inside $S$ according to the following rule which is not canonical: $t_i$'s  are written in  $S$ from  left to right and top to bottom in the entries $(i,j)$ for $i+j<n+2$ if $n$ is even and $i+j\leq n+2$ if $n$ is odd. The position of $\check t_i$'s inside $S$ can be chosen arbitrarily. For instance, for $n=1,2,3,4$ we have:
\[\small \left( {\begin{array}{*{20}{c}}
1&0\\
{{t_2}}&{{\check t}_1}
\end{array}} \right)\,,\,\,\left( {\begin{array}{*{20}{c}}
1&0&0\\
{{t_2}}&{{\check t}_2}&0\\
{{\check t}_4}&{{\check t}_3}&{{\check t}_1}
\end{array}} \right)\,,\,\, \left( {\begin{array}{*{20}{c}}
1&0&0&0\\
{{t_2}}&{{t_3}}&0&0\\
{{t_4}}&{{t_6}}&{{\check t}_2}&0\\
{{t_7}}&{{\check t}_4}&{{\check t}_3}&{{\check t}_1}
\end{array}} \right)\,,\,\, \left( {\begin{array}{*{20}{c}}
1&0&0&0&0\\
{{t_2}}&{{t_3}}&0&0&0\\
{{t_4}}&{{t_5}}&{{\check t}_3}&0&0\\
{{t_7}}&{{\check t}_7}&{{\check t}_5}&{{\check t}_2}&0\\
{{\check t}_9}&{{\check t}_8}&{{\check t}_6}&{{\check t}_4}&{{\check
t}_1}
\end{array}} \right)
.\]
%Here we explain a specific way to choose the independent
%parameters. To do this, first assume that $n$ is an odd integer,
%and define
%\[
%I_1:=\{s_{ij}\, | \, 2\leq i \leq \frac{n+1}{2}\,\ and \,\ 1\leq j
%\leq i\}\,\, \& \,\, I_2:=\{s_{(n+2-i)j}\, | \, s_{ij} \in I_1\}
%\cup \{s_{(n+1)1}\}.
%\]
%Then the set of independent parameters is given by $I=I_1\cup I_2$.
%We denote the independent parameters by $t_2,
%t_3,\ldots,t_{n+1},t_{n+3},\ldots,t_\dt$, once we order them from
%left to right and up to down.
Note that we have already used  $t_1,t_{n+2}$ as coordinate system
of $\Ts$ in Section \ref{section ms1}.
%In the case that $n$ is an even
%integer, we know that $\dt_{n-1}=\dt_n$. Hence we choose $s_{ij}$ as
%an independent parameter, if it is an independent parameter in the
%odd case $n-1$. Below we give examples for $n=1,2,3,4$ respectively:

\begin{prop}
\label{11-01-2016}
The equation $S\Omega S^\tr=\Phi$ yields
\begin{equation}
\label{29/12/2015}
s_{(n+2-i)(n+2-i)}=
%(-1)^{i+1}\frac{1}{\tilde{a}s_{ii}}=
\frac{(-1)^{n+i+1}}{c_n(n+2)^n}\frac{t_1^{n+2}-t_{n+2}}{s_{ii}},
\end{equation}
where $i=1,2,\ldots, \frac{n+1}{2}$ if $n$ is an odd integer, and
$i=1,2,\ldots, \frac{n+2}{2}$ if $n$ is an even integer. Moreover,
one can compute $\check t_i$'s in terms of $t_i$'s. % and $\check t_i\in \C[t_1,t_2,\cdots,t_{\dt},]$
\end{prop}
\begin{proof}
Let us first count the number of equalities that we get from
$S\Omega S^\tr=\Phi$. This is $\frac{(n+1)(n+2)}{2}+1-\dt$.
Note that the left upper triangle of this equality consists of trivial equalities $0=0$.
The equality \eqref{29/12/2015} follows from the $(i,n+2-i)$-th
entry of $S\Omega S^\tr=\Phi$. We have plugged the parameters
$\check t_k=s_{ij}$ inside $S$ such that the equality corresponding
to the $(n+2-j,i)$-th entry of $S\Omega S^\tr=\Phi$ gives us an
equation which computes $\check t_k$ in terms of $\check t_r, \ r<k$
and $t_s$'s. Note that only divisions by $s_{ii}$'s, $t_{n+2}-t_1^{n+2}$ and $t_{n+2}$ occurs.
Another way to see this is to redefine $S:=S^{-1}$ and so we will have the
equality $S\Phi S^\tr=\Omega$.
%$\check t_k$ in this equality . This
%can be computed easily. It is ************.\marginpar{\tiny Complete the proof. $(n+2-j,i)$ does not seem to be true.}
\end{proof}

For $n=1,2,3,4$, we express $\check t_j$'s in terms of $t_i$'s as
follows:\\
$n=1$:
  {\small \[
     {\check t}_1=-\frac{1}{3c_1}(t_1^3-t_3).
  \]}
$n=2$:
  {\small \begin{align}
  &{\check t}_1=\frac{1}{16c_2}(t_1^4-t_4)\ ,\ \ \
  {\check t}_2^2=-\frac{1}{16c_2}(t_1^4-t_4)\, , \label{eq tc2} \\
  &{\check t}_3=\frac{1}{16c_2}(-16c_2t_2{\check t}_2+2t_1^3)\ , \ \ \
  {\check t}_4=\frac{1}{32c_2}(-16c_2t_2^2+t_1^2)\, . \nonumber
  \end{align}}
  $n=3$:
{\small \begin{align}
  &{\check t}_1=-625t_1^5+625t_5\, , \ \ \  {\check t}_2=\frac{625t_1^5-625t_5}{t_3}\, , \nonumber \\
  &{\check t}_3=\frac{-625t_1^5t_2+625t_2t_5-3125t_1^4t_3}{t_3}\, , \ \ \  {\check t}_4=-t_2t_6+t_3t_4-3125t_1^3\, . \nonumber
  \end{align}}
  $n=4$:
{\small \begin{align}\label{eq tc3}
  &{\check t}_1=\frac{t_1^6-t_6}{1296c_4}\ ,\ \ \ \
  {\check t}_2=-\frac{t_1^6-t_6}{1296c_4t_3}\ , \ \ \
  {\check t}_3^2=\frac{t_1^6-t_6}{1296c_4}\ , \ \ \
  {\check t}_4=\frac{t_1^6t_2+9t_1^5t_3-t_2t_6}{1296c_4t_3}\ ,
   \\ \label{eq tc3}
  &
   {\check t}_5=\frac{-432c_4t_5{\check t}_3-t_1^5}{432c_4t_3}, \ \ \ \
  {\check t}_6=\frac{1296c_4t_2t_5{\check t}_3-1296c_4t_3t_4{\check t}_3+3t_1^5t_2+20t_1^4t_3}{1296c_4t_3}\, , \nonumber\\
  &{\check t}_7=\frac{-1296c_4t_5^2-5t_1^4}{2592c_4t_3}\, , \ \ \ \ \nonumber
  {\check t}_8=\frac{1296c_4t_2t_5^2-2592c_4t_3^2t_7-2592c_4t_3t_4t_5+5t_1^4t_2+20t_1^3t_3}{2592c_4t_3}\, ,\nonumber \\
  &{\check t}_9=\frac{-2592c_4t_2t_7-1296c_4t_4^2+t_1^2}{2592c_4}\, .\nonumber
  \end{align}}

%===============================================================================================================================
\section{Gauss-Manin connection} \label{section gmc}
We return to the Gauss-Manin connection $\nabla$, that was
introduced in \eqref{eq gmn}, and we proceed with the computation of
the Gauss-Manin connection matrix $\gmo$, which is given in
\eqref{eq Atilde}.

If we denote by $A(z)$ the Gauss-Manin connection matrix of the
family $\X_{1,z}$ in the basis $\{\eta,{\frac{\partial\eta}{\partial
z}},\ldots, {\frac{\partial^n\eta}{\partial z^n}}\}$, that is
\[
\nabla\left( {\begin{array}{*{20}c}
 {\eta } & {\frac{\partial\eta}{\partial z} } &  \ldots  & {\frac{\partial^n\eta}{\partial z^n} }  \\
\end{array}} \right)^\tr=A(z)dz\otimes
\left( {\begin{array}{*{20}c}
 {\eta } & {\frac{\partial\eta}{\partial z} } &  \ldots  & {\frac{\partial^n\eta}{\partial z^n} }  \\
\end{array}} \right)^\tr\, ,
\]
then we get
\[A(z)=\left( {\begin{array}{*{20}{c}}
0&1&0&0& \ldots &0\\
0&0&1&0& \ldots &0\\
0&0&0&1& \ldots &0\\
 \vdots & \vdots & \vdots & \vdots & \ddots & \vdots \\
0&0&0&0& \ldots &1\\
{{b_1}(z)}&{{b_2}(z)}&{{b_3}(z)}&{{b_4}(z)}& \ldots &{{b_{n +
1}}(z)}
\end{array}} \right)\]
where the functions $b_i(z)$'s are the coefficients of the
Picard-Fuchs equation \eqref{eq pf eta}  associated with the $n$-form $\eta$
that we write in the following format
\[
\frac{\partial^{n+1} \eta}{\partial
z^{n+1}}=b_1(z)\eta+b_2(z)\frac{\partial \eta}{\partial
z}+\ldots+b_{n+1}(z)\frac{\partial^n \eta}{\partial z^n},\ \ \hbox{modulo exact forms}.
\]
 We calculate $\nabla$ with respect to the
basis (\ref{29oct11})
 of $H^n_\dR(\X/\Ts)$.
For this purpose we return back to the one parameter case. For
$z:=\frac{t_{n+2}}{t_1^{n+2}}$,  consider the map
$$
g:\X_{{t_1,t_{n+2}}}\to \X_{1,z},
$$
given by (\ref{20m2014}) with $k=t_1^{-1}$. We have
$g^*\eta=t_1\omega_1$, where by abuse of notation we just write
$\eta=t_1\omega_1$, and
$$
\frac{\partial}{\partial
z}=\frac{-1}{n+2}\frac{t_1^{n+3}}{t_{n+2}}\frac{\partial}{\partial
t_1}.
$$
From these two equalities we obtain the base change matrix $\tilde
S=\tilde S(t_1,t_{n+2})$ such that
$$
\left( {\begin{array}{*{20}c}
 {\eta } & {\frac{\partial\eta}{\partial z} } &  \ldots  & {\frac{\partial^n\eta}{\partial z^n} }  \\
\end{array}} \right)^\tr= \tilde
S^{-1}\left( {\begin{array}{*{20}c}
 {\omega_1 } & {\omega_2} &  \ldots  & {\omega_{n+1} }  \\
\end{array}} \right)^\tr.
$$
Thus we find the Gauss-Manin connection in the basis $\omega_i,\
i=1,2,\ldots,n+1$ as follow:
$$
\gmo=\left (d\tilde S+\tilde S\cdot
A(\frac{t_{n+2}}{t_1^{n+2}})\cdot d(\frac{t_{n+2}}{t_1^{n+2}})\right
)\cdot \tilde S^{-1}.
$$
Let $\gmo[i,j]$ be the $(i,j)$-th entry of the Gauss-Manin
connection matrix $\gmo$.
%If we denote by $\gmo=\left(
%                           \begin{array}{c}
%                             \gmo[i,j] \\
%                           \end{array}
%                         \right)_{1\leq i,j\leq n+1}$ the Gauss-Manin connection matrix ,
We have
\begin{eqnarray}
\label{16/1/2016-1}
\gmo[i,i] &=& -\frac{i}{(n+2)t_{n+2}}dt_{n+2}\, , \,\ 1 \leq i \leq n\, \\
\label{16/1/2016-2}
\gmo[i,i+1]&=& dt_1-\frac{t_1}{(n+2)t_{n+2}}dt_{n+2}\, , \,\ 1 \leq i \leq n\, ,  \\
\gmo[n+1,j]&=&\frac{-S_2(n+2,j)t_1^j}{t_1^{n+2}-t_{n+2}}dt_1+\frac{S_2(n+2,j)t_1^{j+1}}{(n+2)t_{n+2}(t_1^{n+2}-t_{n+2})}dt_{n+2}\, , \,\ 1 \leq j \leq n\, , \nonumber \\
\gmo[n+1,n+1]&=&\frac{-S_2(n+2,n+1)t_1^{n+1}}{t_1^{n+2}-t_{n+2}}dt_1+\frac{\frac{n(n+1)}{2}t_1^{n+2}+(n+1)t_{n+2}}{(n+2)t_{n+2}(t_1^{n+2}-t_{n+2})}dt_{n+2}\,
,\nonumber
\end{eqnarray}
where $S_2(r,s)$ is the Stirling number of the second kind defined in \eqref{16jan2016},
and the rest of the entries of $\gmo$ are zero.
The equalities \eqref{16/1/2016-1} and \eqref{16/1/2016-2} are easy to check, and those with Stirling numbers are checked for  $n=1,2,3,4$.
It would be interesting to prove this statement for arbitrary $n$. We will not need such explicit expressions for the proof of our main theorem.
The Gauss-Manin connection matrix $\gmo$ for $n=1,2$ are
respectively given as follows:
\begin{equation} \small
\gmo_1=\left({\large
              \begin{array}{cc}
                -\frac{1}{3t_3}dt_3 & dt_1-\frac{t_1}{3t_3}dt_3
                \\\\
                -\frac{t_1}{t_1^3-t_3}dt_1+\frac{t_1^2}{3t_3(t_1^3-t_3)}dt_3 &
                -\frac{3t_1^2}{t_1^3-t_3}dt_1+\frac{t_1^3+2t_3}{3t_3(t_1^3-t_3)}dt_3 \\
              \end{array}
            }\right),
            \nonumber
\end{equation}
\begin{equation} \small
\gmo_2=\left({\large
              \begin{array}{ccc}
                -\frac{1}{4t_4}dt_4 & dt_1-\frac{t_1}{4t_4}dt_4 & 0
                \\\\
                0 & -\frac{2}{4t_4}dt_4 & dt_1-\frac{t_1}{4t_4}dt_4
                \\\\
                -\frac{t_1}{t_1^4-t_4}dt_1+\frac{t_1^2}{4t_4(t_1^4-t_4)}dt_4 & -\frac{7t_1^2}{t_1^4-t_4}dt_1+\frac{7t_1^3}{4t_4(t_1^4-t_4)}dt_4 &
                -\frac{6t_1^3}{t_1^4-t_4}dt_1+\frac{3t_1^4+3t_4}{4t_4(t_1^4-t_4)}dt_4 \\
              \end{array}
            }\right) .
            \nonumber
\end{equation}
Let  $\gm$ to be the Gauss-Manin connection matrix of the family
$\X/\T$ written in the basis $\alpha_i,\ i=1,2,\ldots,\alpha_{n+1}$,
i.e., $\nabla \alpha=\gm \alpha$. Then we calculate $\gm$ as follow:
\begin{equation}\label{eq gm 3/27/2016}
\gm=\left (dS+S\cdot \gmo\right )\cdot S^{-1},
\end{equation}
where $S$ is the base change matrix $\alpha=S\omega$.

\section{Proof of Theorem \ref{main3}} \label{section main3}
As we saw in \eqref{eq gm 3/27/2016}, the Gauss-Manin connection
matrix of the family $\X/\T$ in the basis $\alpha$ is given by
\begin{equation}
\gm= dS\cdot S^{-1}+S\cdot \gmo \cdot S^{-1}.
\end{equation}
For a moment, let us consider the entries $s_{ij},\ j\leq i,
(i,j)\not=(1,1)$ of $S$ as independent parameters with only the
following relation:
\begin{equation}
\label{eybivafahalemanbenegar} s_{(n+1)(n+1)}+s_{nn}s_{22}=0.
\end{equation}
We denote by $\tilde\T$ and $\tilde\alpha$ the corresponding family of varieties and a basis
of differential forms. The existence of a vector field $\Ra$  in $\tilde\T$ with the desired
property in relation with the Gauss-Manin connection is equivalent to solve the equation
\begin{equation}
\label{eq dots2} \dot{S}=\Yuk S-S\cdot \gmo(\Ra).
\end{equation}
where $\dot x:=dx(\Ra)$ is the derivation of the function
$x$ along the vector field $\Ra$ in $\tilde\T$. The equalities
corresponding to the entries $(i,j), j\leq i,\ \ (i,j)\not=(1,1)$
serves as the definition of $\dot{s_{ij}}$. The equality
corresponding to $(1,1)$-th and $(1,2)$-th entries give us
respectively
$$
\dot{t}_1=t_3-t_1t_2,\ \
\dot{t}_{n+2}=-(n+2)t_2t_{n+2}.
$$
Recall that $t_2=s_{21}$ and $t_3=s_{22}$. The equalities
corresponding to $(i,i+1)$-th, $i=2,\cdots, n-1$, entries compute
the quantities $\Yuk_i$'s:
\begin{equation}
 \label{11/1/2016}
\Yuk_{i-1}=\frac{t_3s_{ii}}{s_{(i+1)(i+1)}},\ \ i=2,3,\ldots, n-1.
\end{equation}
Finally the equality corresponding to the $(n,n+1)$-th entry is
given by (\ref{eybivafahalemanbenegar}) which is already implemented
in the definition of $\tilde\T$. All the rest are trivial equalities
$0=0$. We conclude the statement of Theorem \ref{main3} for the
moduli space $\tilde\T$.

Now, let us prove the main theorem for the moduli space $\T$. First, note that we have a map
\begin{equation}
 \tilde\T\to \Mat_{(n+1)\times(n+1)}(\C),\ \ (t_1,t_{n+2}, S)\mapsto S \Omega S^{\tr}
\end{equation}
and $\T$ is the fiber of this map over the point $\Phi$. We prove that the vector field $\Ra$ is tangent to the fiber of
the above map over $\Phi$. This follows from
\begin{eqnarray*}
\overbrace{(S \Omega S^{\tr})}^{.}  &=&
\dot S\Omega S^{\tr}+S\dot\Omega S^{\tr}+ S\Omega \dot S^{\tr}\\
&=&
(\Yuk S-S\gmo )\Omega S^{\tr}+S (\gmo \Omega+\Omega\gmo^{\tr}) S^{\tr}+ S\Omega ( S^{\tr}\Yuk^{\tr}-\gmo^{\tr} S^{\tr})\\
&=&
\Yuk \Phi+ \Phi\Yuk^{\tr} \\
&=& 0.
\end{eqnarray*}
where $\dot x:=dx(\Ra)$ is the derivation of the function $x$ along
the vector field $\Ra$ in $\T$.
 The last equality follows from \eqref{11/1/2016} and
Proposition \ref{11-01-2016}.  It follows that if $n$ is an even
integer then $\Yuk_{i-1}=-\Yuk_{n-i},\ i=2,\ldots,\frac{n}{2}$ and
if $n$ is an odd integer then $\Yuk_{i-1}=-\Yuk_{n-i}, \
i=2,\ldots,\frac{n-1}{2}$ and
$$
\Yuk_{\frac{n-1}{2}}=(-1)^{\frac{3n+3}{2}}c_n(n+2)^n\frac{t_3s_{\frac{n+1}{2}\frac{n+1}{2}}^2}{t_1^{n+2}-t_{n+2}}.$$
%In other words
%\begin{equation}
%\label{eq yphi} \Yuk\Phi+ \Phi \Yuk^\tr=0.
%\end{equation}
To prove the uniqueness, first notice that \eqref{11/1/2016}
guaranties the uniqueness of $\Yuk_i$'s. Suppose that there are two
vector fields $\Ra$ and $\hat{\Ra}$ such that
$\nabla_{\Ra}\alpha=\Yuk\alpha$ and
$\nabla_{\hat{\Ra}}\alpha=\Yuk\alpha$. If we set
$\H:=\Ra-\hat{\Ra}$, then
\begin{equation} \label{eq 4114 1}
\nabla_{{\H}}\alpha=0.
\end{equation}
We need to prove that ${\H}=0$, and to do this it is enough to
verify that any integral curve of ${\H}$ is a constant point. Assume
that $\gamma$ is an integral curve of ${\H}$ given as follow
\begin{align}
\gamma:(\C,&\,0) \to \T;\qquad
           x\mapsto \gamma(x). \nonumber
\end{align}
Let us  denote by $\mathcal{C}:=\gamma(\C,0)\subset \T$ the trajectory
of $\gamma$ in $\T$ . We know that the points of $\T$ are
pairs $(X,[\alpha_1,\alpha_2,\ldots,\alpha_{n+1}])$, in
which $X$ is an $n$-fold mirror variety and
$\{\alpha_1,\alpha_2,\ldots,\alpha_{n+1}\}$ is a  basis of
$H^n_\dR(X)$ compatible with its Hodge filtration and has constant
intersection form matrix $\Phi$. Thus, we can parameterize $\gamma$
in such a way that for any $x\in (\C,0)$ the vector field $\H$ on
$\mathcal{C}$ reduces to $\frac{\partial}{\partial x}$, and so, we
have
$\gamma(x)=(X(x),[\alpha_1(x),\alpha_2(x),\ldots,\alpha_{n+1}(x)])$.
We know that $X(x)$ is a member of mirror family that depends only
on the parameter $z$, hence $x$ holomorphically depends to $z$. From
this we obtain a holomorphic function $f$ such that $x=f(z)$. We now
proceed to prove that $f$ is constant. Otherwise, by contradiction
suppose that $f'\neq 0$. Then we get
\begin{equation}\label{eq 72414}
\nabla_{\frac{\partial}{\partial x}}\alpha_1= \frac{\partial
z}{\partial x}\nabla_{\frac{\partial}{\partial z}}\alpha_1.
\end{equation}
Equation (\ref{eq 4114 1}) gives that
$\nabla_{\frac{\partial}{\partial x}}\alpha_1=0$, but since
$\alpha_1=\omega_1$, it follows that the right hand side of (\ref{eq
72414}) is not zero, which is a contradiction. Thus $f$ is constant
and $X(x)$ does not depend on the parameter $x$. Since $X(x)=X$ does
not depends on $x$, we can write the Taylor series of
$\alpha_i(x),\,\ i=1,2,3,\ldots, n+1,$ in $x$ at some point $x_0$ as
$\alpha_i(x)=\sum_j (x-x_0)^j\alpha_{i,j},$ where $\alpha_{i,j}$'s
are elements in $H_\dR^n(X)$ independent of $x$. In this way the
action of $\nabla_{\frac{\partial}{\partial x}}$ on $\alpha_i$ is
just the usual derivation $\frac{\partial}{\partial x}$. Again
according to (\ref{eq 4114 1}) we get
$\nabla_{\frac{\partial}{\partial x}}\alpha_i=0$, and we conclude
that $\alpha_i$'s also do not depend on $x$. Therefore, the image of
$\gamma$ is a point.\hfill\(\square\)\\

The modular vector field $\Ra$ for $n=1, 2, 3, 4$, are given as
follows:\\
$n=1$:
  \begin{equation} \label{eq modvec 1}
\Ra_{1}:\left \{ \begin{array}{l}
\dot{t}_1=\frac{1}{3c_1}(-3c_1t_1t_2-(t_1^3-t_3))\\
\dot{t}_2=\frac{1}{9c_1^2}(t_1(t_1^3-t_3)-9c_1^2t_2^2)\\
\dot{t}_3=-3t_2t_3
\end{array} \right..
\end{equation}
$n=2$: We know that $\dim \T_2=3$, hence the modular vector
  field $\Ra_{2}$ should have three components, but to avoid the
  second root of $\check t_2$ that comes from \eqref{eq tc2} we add
  one more variable $t_3:=\check t_2$. Thus we find $\Ra_{2}$ as
  follow:
  \begin{equation} \label{eq modvec 2}
\Ra_{2}:\left \{ \begin{array}{l} \dot{t}_1=-t_1t_2+t_3\\
\dot{t}_2=-\frac{1}{32c_2}(t_1^2+16c_2t_2^2)\\
\dot{t}_3=-\frac{1}{8c_2}(16c_2t_2t_3+t_1^3)\\  \dot{t}_4=-4t_2t_4
\end{array} \right.,
\end{equation}
such that the following equation holds among $t_i$'s
\begin{equation}\label{eq t3}
 t_3^2=-\frac{1}{16c_2}(t_1^4-t_4).
\end{equation}

$n=3$: The vector field $\Ra_3$ has been calculated in \cite{ho22}, but in a different chart. $\Ra_3$ in the chart chosen in this paper is as follow:
\begin{equation}  \label{eq modvec 3}
\Ra_{3}:\left \{ \begin{array}{l}
\dot{t}_1=t_3-t_1t_2\\
\dot{t}_2=\frac{t_3^3t_4-5^4t_2^2(t_1^5-t_5)}{5^4(t_1^5-t_5)}\\
\dot{t}_3=\frac{t_3^3t_6-3\times5^4t_2t_3(t_1^5-t_5)}{5^4(t_1^5-t_5)} \\
\dot{t}_4=-t_2t_4-t_7\\
\dot{t}_5=-5t_2t_5\\
\dot{t}_6=5^5t_1^3-t_2t_6-2t_3t_4\\
\dot{t}_7=-5^4t_1t_3-t_2t_7 \\
\end{array} \right.\,.
\end{equation}

$n=4$: Similar to the case $n=2$ and in order to avoid the
  second root of $\check t_3$ given in \eqref{eq tc3}, we add the variable $t_8:=\check t_3$
  and we find:
  \begin{equation}  \label{eq modvec 4}
 \Ra_{4}:\left \{ \begin{array}{l}
\dot{t}_1=t_3-t_1t_2\\
\dot{t}_2=\frac{1296c_4t_3^2t_4t_8-t_1^6t_2^2+t_2^2t_6}{t_1^6-t_6}\\
\dot{t}_3=\frac{1296c_4t_3^2t_5t_8-3t_1^6t_2t_3+3t_2t_3t_6}{t_1^6-t_6} \\
\dot{t}_4=\frac{-1296c_4t_3^2t_7t_8-t_1^6t_2t_4+t_2t_4t_6}{t_1^6-t_6}\\
\dot{t}_5=\frac{1296c_4t_3t_5^2t_8-4t_1^6t_2t_5-2t_1^6t_3t_4+5t_1^4t_3t_8+4t_2t_5t_6+2t_3t_4t_6}{2(t_1^6-t_6)}\\
\dot{t}_6=-6t_2t_6\\
\dot{t}_7=\frac{1296c_4t_4^2-t_1^2}{2592c_4} \\
\dot{t}_8=\frac{-3t_1^6t_2t_8+3t_1^5t_3t_8+3t_2t_6t_8}{t_1^6-t_6}
\end{array} \right.,
\end{equation}
where
\begin{equation}\label{eq t8}
 t_8^2=\frac{1}{1296c_4}(t_1^6-t_6).
\end{equation}
In this case the functions $\Yuk_1$ and $\Yuk_2$ are given by
\begin{equation} \label{eq yukawa 4}
\Yuk_1^2=(-\Yuk_2)^2=\frac{1296c_4t_3^4}{t_1^6-t_6}.
\end{equation}
%The set of singularities contains
%\[
%t_6=t_3-t_1t_2=6^4c_4t_4^2-t_1^2=t_8-6^4c_4t_4^3=t_5-3t_1t_4=-t_4^2-t_2t_7=0.
%\]

%====================================================================================
\section{Enumerative properties of $q$-expansions} \label{section ep}
In order to find the $q$-expansion of a solution of $\Ra$, we follow
the process given in \cite[\S~5.2]{GMCD-MQCY3} for the case $n=3$.
Consider the vector field $\Ra$ as follow
\begin{equation} \label{eq modvec}
\Ra:\left \{ \begin{array}{l}
\dot{t}_1=f_1(t_1,t_2,\ldots,t_\dt)\\
\dot{t}_2=f_2(t_1,t_2,\ldots,t_\dt)\\
\vdots \\
\dot{t}_\dt=f_\dt(t_1,t_2,\ldots,t_\dt)
\end{array} \right.,
\end{equation}
where for $1\leq j \leq \dt$,
$$
f_j\in \C[t_1,t_2,\ldots, t_{\dt},\frac{1}{t_{n+2}(t_{n+2}-t_1^{n+2})\check t %\prod_{i=1}^s t_{j_i}
 }],
$$
and $\check t$ is the same as in Theorem \ref{main3}. Let us assume
that
$$
t_j=\sum_{k=0}^\infty t_{j,k}q^k,\ \ j=1,2,\ldots,\dt,
$$
form a solution of $\Ra$, where $t_{j,k}$'s are subject to be
constants, and $\dot \ast=a\cdot q\cdot \frac{\partial
\ast}{\partial q}$, where $a$ is an unknown constant.
%Thus we
%have
%\begin{equation}\label{eq tdot}
%\dot{t}_j=\sum_{k=1}^{\infty}at_{j,k}q^k,\qquad j=1,2,\ldots,\dt.
%\end{equation}
By comparing the coefficients of $q^k, k\geq 2$ in both sides of
(\ref{eq modvec})  we find recursions for $t_{j,k}$'s. Let
$$
p_k:=(t_{1,k},t_{2,k},\ldots,t_{\dt,k})\, , \,\,\, k=1,2,3,\ldots.
$$
By comparing the coefficients of $q^0$ we get that $p_0$ is a
singularity of $\Ra$. The same for $q^1$, gives us some constrains
on $t_{j,1}$. Therefore, some of the coefficients $t_{j,k}$, for finite number of $j$ and $k$,  are free
initial parameters of the recursion and we have to fix them by other means.
%After finding the solutions, we proceed with the study of the
%enumerative properties of the $q$-expansions. Following we state the
%results in the cases $n=1,2,4$.

%--------------------------------------------------------------------------------------------------------------------------------------------
\subsection{The case $n=1$} \label{subsection epk3}

Considering the modular vector field $\Ra_1$ given in \eqref{eq
modvec 1}, we find $Sing(\Ra_1)=Sing_1\cup Sing_2$, where
\begin{align}
&Sing_1: t_2=t_1^3-t_3=0, \nonumber\\
&Sing_2: t_3=t_1^2+3c_1t_2=0. \nonumber
\end{align}
Thus we get
\begin{align}
&p_0=(t_{1,0},-\frac{1}{3c_1}t_{1,0}^2,0)\in Sing_2.\nonumber
\end{align}
The comparison of the coefficients of $q^1$  gives us
$a=\frac{1}{c_1}t_{1,0}^2$ and
\begin{align}
&p_1=(\frac{2}{9}\frac{t_{3,1}}{t_{1,0}^2},-\frac{1}{27c_1}\frac{t_{3,1}}{t_{1,0}},t_{3,1}).\nonumber
\end{align}
If we choose $c_1=3^{-3}$, $t_{1,0}= \frac{1}{3}$ and $t_{3,1}=1$,
then we find the solution given in \eqref{eq solution R01} for
$\Ra_1$.

The coefficient of $q^k,\, k=0,1,2,3,\ldots$, in
$\theta_3(q^{2r})\theta_3(q^{2s}),\, r,s\in \N$, gives the number of
integer solutions of the equation $rx^2+sy^2=k$, where $x$ and $y$
are unknown variables. Therefore

%\begin{proof}
%We know that $\theta_3(q^2)=1+2\sum_{j=1}^{\infty}q^{j^2}$, hence
%\begin{equation} \label{eq thth}
%\theta_3(q^{2r})\theta_3(q^{2s})=1+2\sum_{i=1}^{\infty}q^{ri^2}+2\sum_{j=1}^{\infty}q^{sj^2}+4\sum_{i,j=1}^{\infty}q^{ri^2+sj^2}.
%\end{equation}
%If $(i,0)$ or $(0,j)$ is a solution, then $(-i,0)$ or $(0,-j)$,
%respectively, is another solution, and if $(i,j)$, with $i\neq 0,\
%j\neq 0$, is a solution, then $(-i,j)$, $(i,-j)$ and $(-i,-j)$ are
%other solutions. Therefore, because of the above fact, the proof
%follows from the equation (\ref{eq thth}).
%\end{proof}

\begin{prop}\label{coro ep1}
The coefficient of $q^k,\, k=0,1,2,3,\ldots$, in
$\theta_3(q^2)\theta_3(q^6)$ gives the number of integer solutions
of equation $x^2+3y^2=k$.
\end{prop}
For more information about the number of integer solutions of
equation $x^2+3y^2=k$ see \cite[A033716]{Oeis} and
the references therein. As we saw in \eqref{eq solution R01},
${t}_1(q)=\frac{1}{3}(2\theta_3(q^2)\theta_3(q^6)-\theta_3(-q^2)\theta_3(-q^6))$.
If we denote by $t_1(q):=\sum_{k=0}^\infty t_{1,k}q^k$, then in the
following proposition we state enumerative properties of $t_{1,k}$.
\begin{prop}
Let $k$ be a non-negative integer. If $k=4m$ for some $m\in \Z$,
then the equation $x^2+3y^2=k$ has $3t_{1,k}$ integer solutions,
otherwise the equation has $t_{1,k}$ integer solutions.
\end{prop}
\begin{proof}
Suppose that $\theta_3(q^2)\theta_3(q^6)=\sum_{k=0}^\infty a_{k}q^k$
and $\theta_3(-q^2)\theta_3(-q^6)=\sum_{k=0}^\infty b_{k}q^k$. Fix a
non-negative integer $k$. If $k=4m$ for some $m\in \Z$,  then
$a_k=b_k$, otherwise $a_k=-b_k$. This fact together with Proposition
\ref{coro ep1} complete the proof.
\end{proof}
Y. Martin in \cite{yvma} studied a more general class of
$\eta$-quotients. By definition an \emph{$\eta$-quotient} is a
function $f(q)$ of the form $f(q)=\prod_{j=1}^s\eta^{r_j}(q^{t_j})$,
where $t_j$'s are positive integers and $r_j$'s are arbitrary
integers. He gives an explicit finite classification of modular
forms of this type which is listed in \cite[Table~I]{yvma}. In
\eqref{eq solution R01} we found
\begin{equation} \label{eq t3ec}
{t}_3(q)=\frac{\eta^9(q^3)}{\eta^3(q)},
\end{equation}
which is the multiplicative $\eta$-quotient $\sharp 3$ presented by
Y. Martin in Table I of \cite{yvma}. For more details and references
about this $\eta$-quotient the reader is referred to \cite[A106402]{Oeis}. Finally, note that  if we define $\sum_{k=0}^\infty
a_kq^k:={t}_2(q)=\frac{1}{8}(E_2(q^2)-9E_2(q^6)),$ then one can see
that $3\mid a_k$ for integers $k\geq 1$.

%---------------------------------------------------------------------------------------------------------------------------------------------------------------
\subsection{The case $n=2$}
From \eqref{eq modvec 2} we get
\[
Sing(\Ra_2)=\{(t_1,t_2,t_3,t_4)\, |\,
t_4=t_3-t_1t_2=t_1^2+16c_2t_2^2=0\},
\]
hence we find
\begin{align}
&p_0=(t_{1,0},\frac{1}{4}\,k_0t_{1,0},\frac{1}{4}\,k_0t_{1,0}^2,0)\in
Sing(\Ra_2),\nonumber
\end{align}
where $k_0=\frac{1}{\sqrt{-c_2}}$. Comparing  the
coefficients of $q^1$, we get $a=-t_{1,0}k_0$ and
\begin{align}
&p_1=(-\frac{6}{5}\frac{t_{3,1}}{k_0t_{1,0}},\frac{1}{10}\frac{t_{3,1}}{t_{1,0}},t_{3,1},-\frac{64}{5}\frac{t_{1,0}^2t_{3,1}}{k_0}),\nonumber
\end{align}
where the equality
$t_{4,1}=-\frac{64}{5}\frac{t_{1,0}^2t_{3,1}}{k_0}$ follows from
\eqref{eq t3}. We set  $c_2=-\frac{1}{64}$, $t_{1,0}=
\frac{1}{40}$ and $t_{3,1}=-1$ and  find the solution given in
\eqref{eq solution R02} for $\Ra_2$.
{\tiny
\begin{table}
\centering \small
\begin{tabular} {|c|c|c|c||c|c|c|c|}
  \hline
   $\Ra_1$ & $t_1$ & $t_2$ & $t_3$ & $\Ra_2$ & $\frac{10}{6}{t}_1(\frac{q}{10})$ & $\frac{10}{4}{t}_2(\frac{q}{10})$ & $10^4{t}_4(\frac{q}{10})$ \\ \hline
  $q^0$ & 1/3 & -1 & 0 & $q^0$ & 1/24 & 1/8 & 0 \\
  $q^1$ & 2 & -3 & 1 & $q^1$ & 1 & -1 & 1 \\
  $q^2$ & 0 & -9 & 3 & $q^2$ & 1 & -5 & -8 \\
  $q^3$ & 2 & 15 & 9 & $q^3$ & 4 & -4 & 12 \\
  $q^4$ & 2 & -21 & 13 & $q^4$ & 1 & -13 & 64 \\
  $q^5$ & 0 & -18 & 24 & $q^5$ & 6 & -6 & -210 \\
  $q^6$ & 0 & 45 & 27 & $q^6$ & 4 & -20 & -96  \\
  $q^7$ & 4 & -24 & 50 & $q^7$ & 8 & -8 & 1016  \\
  $q^8$ & 0 & -45 & 51 & $q^8$ & 1 & -29 & -512 \\
  $q^9$ & 2 & 69 & 81 & $q^9$ & 13 & -13 & -2043 \\
  $q^{10}$ & 0 & -54 & 72 & $q^{10}$ & 6 & -30 & 1680 \\
  $q^{11}$ & 0 & -36 & 120 & $q^{11}$ & 12 & -12 & 1092 \\
  $q^{12}$ & 2 & 105 & 117 & $q^{12}$ & 4 & -52 & 768 \\
  $q^{13}$ & 4 & -42 & 170 & $q^{13}$ & 14 & -14 & 1382 \\
  $q^{14}$ & 0 & -72 & 150 & $q^{14}$ & 8 & -40 & -8128 \\
  $q^{15}$ & 0 & 90 & 216 & $q^{15}$ & 24 & -24 & -2520  \\
      \hline
\end{tabular}
\caption{Coefficients of $q^k,\,\, 0\leq k \leq 15,$ in the
q-expansion of the solutions of $\Ra_1$ and $\Ra_2$.} \label{table1}
\end{table}}

The sum of positive odd divisors of a positive integer $k$, which is
also known as \emph{odd divisor function}, was introduced by
Glaisher \cite{glaisher} in 1906. Let $k$ be a positive integer. We
denote the sum of divisors, the sum of odd divisors and the sum of
even divisors of $k$, by $\sigma(k)$, $\sigma^o(k)$ and
$\sigma^e(k)$, respectively, i.e.,
\[
\sigma(k)=\sum_{d\mid k}d\,\,\,\,\ \&\,\,\,\,\
\sigma^o(k)=\sum_{\mathop {d\mid k}\limits_{\textrm{d is odd}}}d
\,\,\,\,\ \&\,\,\,\,\ \sigma^e(k)=\sum_{\mathop {d\mid
k}\limits_{\textrm{d is even}}}d\, .
\]
We have $\sigma(k)=\sigma^o (k)+\sigma^e(k)$ and $\sigma^o(k)=\sigma(k)-2\sigma(k/2)$,
where $\sigma(k/2):=0$ if $k$ is an odd integer. It follows from \eqref{eq
solution R02} that $t_1$ is the generating function of the odd divisor function:
\[
\frac{10}{6}{t}_1(\frac{q}{10})=\sum_{k=0}^\infty\sigma^o(k)q^k=\frac{1}{24}(\theta_3^4(q^2)+\theta_2^4(q^2)),
\]
where by definition $\sigma^o(0)=1/24$.  For more details about the odd divisor function see
\cite[A000593]{Oeis}.

Comparing the coefficients of $t_2$ presented in Table
\ref{table1} with the integers sequence given in the
\cite[A215947]{Oeis} we find that
\[
\sum _{k=0}^\infty
(\sigma^o(2k)-\sigma^e(2k))q^k=\frac{10}{4}{t}_2(\frac{q}{10})=\frac{1}{24}(E_2(q^2)+2E_2(q^4))
\]
where we define $\sigma^o(0)-\sigma^e(0):=1/8$. % (we verified this for the first 100 coefficients).

%Thus we get following proposition.
%\begin{prop}
% $t_2$ generates the function of difference between the sum of the odd divisors and the sum of the even divisors of $2k$, i.e., $\sigma^o(2k)-\sigma^e(2k)$.
%\end{prop}

Another nice observation is about
$10^4{t}_4(\frac{q}{10})=\eta^8(q)\eta^8(q^2)$. The same as $t_3$ in
the case of elliptic curve, see (\ref{eq t3ec}), we see that $t_4$
is the $\eta$-quotient $\sharp 2$ classified by Y. Martin in Table I
of \cite{yvma}, see also \cite[A002288]{Oeis} and the references therein. It is worth to point out that this
$\eta$-quotient appears in the work of Heekyoung Hahn \cite{heeha}.
She proved that $3\mid \mu_{3k}$, $k=0,1,2,\ldots$, where $\mu_k$ is
defined as follow
\[
\sum_{k=0}^\infty \mu_kq^k:=\eta^8(q)\eta^8(q^2).
\]
She also found some partition congruences by using the notion of
colored partitions, for more details see \cite[~ \S 6]{heeha}.

%---------------------------------------------------------------------------------------------------------------------------------------------------------------
\subsection{The case $n=4$}

{\tiny
\begin{table}[b]
\centering \tiny
\begin{tabular}{|r|c|c|c|c|c|c|c|}
         \hline
         $\Ra_4$ & $q^0$ & $q^1$ & $q^2$ & $q^3$ & $q^4$ & $q^5$ & $q^6$ \\
         \hline
         $\frac{1}{20}t_1$ & $\frac{1}{720}$ & 1 & 4131 & 51734044 & 918902851011 & 19562918469120126 & 465569724397794578388 \\
         \hline
         $\frac{1}{216}t_2$ & $-\frac{1}{216}$ & 9 & 110703 & 2248267748 & 55181044614231 & 1498877559908208054 & 43378802521495632926652 \\
         \hline
         $\frac{1}{14}t_3$ & $-\frac{1}{504}$ & 11 & 115137 & 2265573692 & 54820079452449 & 1477052190387154386 & 42523861222488896739828 \\
         \hline
         $\frac{1}{24}t_4$ & $-\frac{1}{144}$ & 16 & 193131 & 3904146832 & 95619949713765 & 2594164605185043648 & 75018247757143686903060 \\
         \hline
         $\frac{1}{2}t_5$ & $-\frac{1}{144}$ & 45 & 469872 & 9215455916 & 222628516313454 & 5992746995783064438 & 172421735348939185816992 \\
         \hline
         $-6^6t_6$ & 0 & -1 & 1944 & 10066356 & 139857401664 & 2615615263199250 & 57453864811412558112 \\
         \hline
         $-\frac{1}{2}t_7$ & $-\frac{1}{72}$ & 7 & 32859 & 414746092 & 7395891627375 & 157811370338782458 & 3761184845284146266940 \\
         \hline
         $\frac{18}{7}t_8$ & $-\frac{1}{3024}$ & 7 & 54855 & 1034706148 & 24546181658391 & 653902684588247058 & 18687787944102314534628 \\
         \hline
         \end{tabular}
\caption{Coefficients of $q^k,\,\, 0\leq k \leq 6,$ in the
q-expansion of a solution of $\Ra_4$.} \label{table2}
\end{table}
}
The set of the singularities of $\Ra_4$ contains the set of
$(t_1,t_2,\ldots,t_8)$'s that satisfy
\begin{equation}\label{eq sing R4}
t_6=t_3-t_1t_2=6^4c_4t_4^2-t_1^2=t_8-6^4c_4t_4^3=t_5-3t_1t_4=-t_4^2-t_2t_7=0.
\end{equation}
Hence if we fix $t_{1,0}$ and $t_{2,0}$, then from \eqref{eq sing
R4} we get
\[
p_0=(t_{1,0},t_{2,0},t_{1,0}t_{2,0},-\frac{1}{36k_0}t_{1,0},-\frac{1}{12k_0}t_{1,0}^2,0,-\frac{1}{1296c_4}\frac{t_{1,0}^2}{t_{2,0}},-\frac{1}{36k_0}t_{1,0}^3),
\]
where $c_4=k_0^2$. By comparing coefficients of $q^1$ we find
\[
a=-6t_{2,0},
\]
and
\[
p_1=(\frac{60k_0t_{8,1}}{49t_{1,0}^2},\frac{-162k_0t_{2,0}t_{8,1}}{49t_{1,0}^3},\frac{-66k_0t_{2,0}t_{8,1}}{7t_{1,0}^2},\frac{16t_{8,1}}{147t_{1,0}^2},
\frac{45t_{8,1}}{49t_{1,0}},\frac{3888k_0t_{1,0}^3t_{8,1}}{49},\frac{t_{8,1}}{1512k_0t_{1,0}t_{2,0}},t_{8,1}).
\]
After fixing $k_0=6^{-3}$, $t_{1,0}=\frac{1}{36}$, $t_{2,0}=-1$ and
$t_{8,1}=\frac{49}{18}$ we find the $q$-expansion of a solution of
$\Ra_4$. We list the first seven coefficients of $q^k$'s in Table
\ref{table2}. As it was expected, after multiplying $t_j$'s by a
constant, all the coefficients are integers. If we compute the $q$-expansion of $\Yuk_1^2$ given in \eqref{eq
yukawa 4}, then we find
{\small
\begin{align}
\frac{1}{6}\Yuk_1^2=6&+120960\,q+4136832000\,q^2+148146924602880\,q^3+5420219848911544320\,q^4\nonumber\\
&+200623934537137119778560\,q^5+7478994517395643259712737280\,q^6\nonumber\\
&+280135301818357004749298146851840\,q^7+10528167289356385699173014219946393600\,q^8\nonumber\\
&+396658819202496234945300681212382224722560\,q^9\nonumber\\
&+14972930462574202465673643937107499992165427200\,q^{10}+\ldots
\nonumber
\end{align}}
which is the $4$-point function discussed in \cite[Table~1,
$d=4$]{GMPR}.  We have also computed the $q$-expansion of the modular coordinate $z$
\begin{equation}
\frac{z}{6^6}=\frac{t_6}{(6t_1)^6}=q
-6264q^2
-8627796q^3
-237290958144q^4
-4523787606611250q^5
%-101677347076292728992q^6
+\cdots
\end{equation}
which coincides with the one computed in \cite[\S 6.1]{KlemmPandharipande2008}.
The computation of genus $1$ topological string partition function ${\cal F}_1$ in \cite[\S 6.1 ]{KlemmPandharipande2008} may offer further evidences  that our computation of $\Ra_4$ is correct.
\href{http://w3.impa.br/~hossein/WikiHossein/files/Singular%20Codes/2016-03-GMCD-DF-SingularCode.txt}
{For the computer codes used in this article see the first author's
webpage} or the buttom of the tex file of the present article  in {\tt arxiv.org}.

\newcommand{\etalchar}[1]{$^{#1}$}
\def\cprime{$'$} \def\cprime{$'$} \def\cprime{$'$} \def\cprime{$'$}

%\bibliography{../Biblio/biblio}
%\bibliographystyle{alpha}

%}
%%\begin{thebibliography}{99}

%\bibliographystyle{plain}
%\bibitem{cand2} P. Candelas, X.C. de la Ossa, P.S. Green, and L. Parkes,
%{\em A pair of Calabi-Yau manifolds as an exactly soluble
%superconformal theory}, Nuclear Phys. B {\bf 359} (1991), 21-74.

%\end{thebibliography}
\end{document}